\documentclass[12pt,a4paper]{amsart}

\title[BMM Symmetrising Trace Conjecture]{The BMM Symmetrising Trace Conjecture for Families of Complex Reflection Groups of Rank Two}
\author{Eirini Chavli}
\address{E.C.: Institute of Algebra and Number Theory, University of Stuttgart, Stuttgart, Germany}
\email{eirini.chavli@mathematik.uni-stuttgart.de}

\author{G\"otz Pfeiffer}
\address{G.P.: School of Mathematical and Statistical Sciences,
	University of Ireland, Galway, Ireland}
\email{goetz.pfeiffer@universityofgalway.ie}

\makeatletter
\@namedef{subjclassname@2020}{%
	\textup{2020} Mathematics Subject Classification}
\makeatother

\subjclass[2020]{Primary 20C08; Secondary 20F55}

\usepackage[euler-digits]{eulervm}
\usepackage[margin=1in]{geometry}
\usepackage{hyperref}
\usepackage{enumitem}
\usepackage{blkarray}
\usepackage{array}
\usepackage{mathdots}

\usepackage[euler-digits]{eulervm}
\usepackage[margin=1in]{geometry}

\usepackage{tikz}
\tikzset{every node/.style={circle,draw=none,minimum size=4mm,inner sep=0pt}}
\tikzset{r/.style={magenta,->,>=latex}}%
\tikzset{b/.style={cyan,->,>=latex}}%
\tikzset{g/.style={olive,->,>=latex,very thin}}%

\newcommand{\Span}[1]{\langle #1 \rangle}

\usepackage{tikz}
\tikzset{every node/.style={circle,draw=none,minimum size=4mm,inner sep=0pt}}
\tikzset{r/.style={red,->,>=latex}}%
\tikzset{b/.style={blue,->,>=latex}}%

\numberwithin{equation}{section}

\newtheorem{Theorem}{Theorem}[section]
\newtheorem{Proposition}[Theorem]{Proposition}
\newtheorem{conj}[Theorem]{Conjecture}
\newtheorem{rem}[Theorem]{Remark}
\newtheorem{ex}[Theorem]{Example}
\newtheorem{lem}[Theorem]{Lemma}
\newtheorem{cor}[Theorem]{Corollary}
\newtheorem{defi}[Theorem]{Definition}

\begin{document}
	%!TeX spellcheck = en_GB
	\begin{abstract}
		The exceptional complex reflection groups of
		rank 2 are partitioned into three families.
		We construct explicit matrix models for the Hecke algebras associated to the maximal groups in the tetrahedral and octahedral family,
		and use them to verify the BMM symmetrising trace conjecture for all
		groups in these two families, providing evidence that
		a similar strategy might apply for the icosahedral
		family.

	\end{abstract}

	\maketitle

\section{Introduction}
Iwahori--Hecke algebras associated to Weyl groups arise as
endomorphism rings of certain representations of finite reductive
groups and they are used to study their representation theory. The
structure of these algebras is determined in terms of generators and
relations, and it turns out to be a deformation of the group algebra
of the corresponding Weyl group.  More generally, the Coxeter
presentation of any real reflection group can be deformed in this way
to define an Iwahori--Hecke algebra.  In fact, one can further
generalise this construction to all complex reflection groups.

In 1993 in a conference on a Greek island called Spetses, Brou\'e,
Malle and Michel \cite{BMM} initiated the study of some objects
attached to complex reflection groups, which they named \emph{Spetses}
after the island. Considering these objects, complex reflection groups
play for Spetses the role Weyl groups play for finite reductive
groups.  In 1998, motivated by the definition of Iwahori-Hecke
algebras for real reflection groups, Brou\'e, Malle and Rouquier
\cite{BMR} associated to every complex reflection group a generic
Hecke algebra.

Inspired by the observation that many examples of generic Hecke
algebras share the nice properties of Iwahori--Hecke algebras,
Brou\'e, Malle, Michel and Rouquier stated a list of conjectures.
Among them there are the following two fundamental conjectures, which
are extremely important and often assumed as hypothesis, in the study
of the structure and representation theory of generic Hecke algebras,
as well as of other related structures such as Cherednik algebras:
\begin{itemize}
\item the ``BMR freeness conjecture'' \cite{BMR} states that the
  generic Hecke algebra is a free module of rank equal to the order of
  the associated complex reflection group;
\item the ``BMM symmetrising trace conjecture'' \cite{BMM} states that
  the Hecke algebra admits a \emph{canonical} symmetrising trace.
\end{itemize}

Complex reflection groups have been classified by Shephard and
Todd~\cite{ShTo} as an infinite family of groups called $G(de, e, n)$,
and a family of $34$ exceptional groups called $G_4, \dots, G_{37}$.
This classification allows us to try and verify the conjectures
case-by-case.

The BMR freeness conjecture was a known fact for the groups
$G(de, e, n)$. Between 2012 and 2017, there was outstanding
developments regarding this conjecture, which is now proved for all
the exceptional complex reflection groups. A detailed list of the
state-of-the art of the proof can be found in \cite[Theorem
3.5]{BCCK}.

The second conjecture was also known to hold for the groups
$G(de,e,n)$ \cite{BreMa, MM98}. However, there had been until 2018 a
little progress in the proof of this conjecture for the exceptional
complex reflection groups; the conjecture was established in
\cite{MM10} only for the groups $G_4$, $G_{12}$, $G_{22}$ and $G_{24}$
(the case of $G_4$ was later independently checked in \cite{Mar46}).

In order to understand the difficulties for the verification of the
BMM symmetrising trace conjecture, we give the definition of a
canonical symmetrising trace. Let $W$ be a complex reflection group
and let $H(W)$ be the associated generic Hecke algebra, defined over a
Laurent polynomial ring $R(W)$.  A \emph{symmetrising trace} on $H(W)$ is a
linear map $\tau : H(W) \rightarrow R(W)$ such that
$\tau(hh') = \tau(h'h)$ for all $h, h' \in H(W)$ and the bilinear form
$(h, h') \mapsto \tau(hh')$ is non-degenerate.  The symmetrising trace
$\tau$ is \emph{canonical}, when it satisfies two further conditions
\cite[2.1]{BMM}, which imply that $\tau$ is unique.

Let $\mathcal{B}$ be a basis for $H(W)$ as an $R(W)$-module.  Then a
linear map $\tau \colon H(W) \rightarrow R(W)$ is a symmetrising trace
if the Gram matrix $A:=(\tau(bb'))_{b,b' \in \mathcal{B}}$ is symmetric and
its determinant is a unit in $R(W)$.  If $1_{H(W)} \in \mathcal{B}$
then we can define a linear map $\tau_{\mathcal{B}}\colon H(W) \rightarrow R(W)$ by
setting $\tau_{\mathcal{B}}(1_{H(W)}) = 1$ and $\tau_{\mathcal{B}}(b) = 0$ for $b \neq 1_{H(W)}$.
In the real case, with respect to the \emph{standard basis} of $H(W)$, this
map $\tau_{\mathcal{B}}$ is the canonical symmetrising trace.

In the complex case, the proof of the BMR freeness conjecture includes
the construction of bases. More precisely, the first author has
constructed in \cite{Ch17, Ch18} explicit bases for the generic Hecke
algebras associated with the groups $G_4,\ldots,G_{16}$, which are of
a ``nice'' form: they have an inductive nature involving powers of a
central element $z$ and they also have a \emph{parabolic
  structure}. Moreover, $1_{H(W)}$ is always a basis element. We call
such a basis a \emph{$z$-basis}.

In 2018, the first author together with Boura, Chlouveraki and
Karvounis \cite{BCC, BCCK} used $z$-bases $\mathcal{B}$ for the cases
of $G_4$, \ldots, $G_8$, and $G_{13}$.  The corresponding linear map
$\tau_{\mathcal{B}}$ sends any element $h \in H(W)$ to the coefficient
of $1_{H(W)}$ when $h$ is expressed as an $R(W)$-linear combination of
the elements of $\mathcal{B}$.  In order to find the Gram matrix $A$,
we therefore need to express products of the form $b b'$,
$b, b' \in \mathcal{B}$, as linear combinations of the basis
$\mathcal{B}$.  Even if this problem is a standard task of linear
algebra, its solution can be quite challenging when the coefficients
are multivariate polynomials in many dimensions.
The first author together with Boura, Chlouveraki and Karvounis managed to overcome these difficulties, by using the nice form of the $z$-basis
to reduce the task of finding $A$ to the computation of only a few products
$bb'$ with a program written in the language \texttt{C++}.

The cases $G_4$, \ldots, $G_8$, and $G_{13}$ did involve a large amount of
ad hoc calculations made by hand, which seemed unfeasible in other
cases.  This is due to a number of factors like the size of the group
$W$, the size of its centre, and the number of indeterminates that
appear in the presentation of the generic algebra $H(W)$.

In 2021 in our paper \cite{ChP}, we described the centre of the
generic Hecke algebra associated to the exceptional groups for which
we knew the validity of both conjectures, namely the groups
$G_4$,\ldots, $G_8$, $G_{12}$, $G_{13}$, and $G_{22}$.  Here, for each
group $W$, we fix a parabolic subgroup $W'$ and construct a $z$-basis
from a choice of coset representatives of $W'$ in~$W$.  Relative to
this basis we then obtain a faithful representation of $H(W)$, where
each generator is represented by a matrix.  This approach allowed us
to create a computer program in GAP3, with all calculations being
automated. As a result, we can express any element of $H(W)$ as a
linear combination of a given, and hence every, basis. Using this
program, we recover for these cases the validity of the BMM
symmetrising trace conjecture.

For larger groups, the following problems arise: Firstly, not every
choice of coset representatives yields a $z$-basis, and it is not at
all obvious how to make the right choice.  And secondly, even if we
have found such a $z$-basis $\mathcal{B}$, the trace function
$\tau_{\mathcal{B}}$ is not necessarily symmetric.
In this paper, we overcome these difficulties and we prove
the following:
\begin{Theorem}
  The BMM symmetrising trace conjecture holds for the exceptional  complex reflection groups $G_4, \dots, G_{15}$.
\end{Theorem}

The aforementioned groups are complex reflection groups of rank $2$.
For each
exceptional group $W$ of rank $2$, its central quotient $W/Z(W)$ is
the rotation group of a platonic solid, yielding three disjoint
families:
\begin{itemize}
\item The \emph{tetrahedral} family, where $W/Z(W)$ is the alternating
  group $\mathfrak{A}_4$, consists of the exceptional groups $G_4$, \ldots, $G_7$.
\item The \emph{octahedral} family, where $W/Z(W)$ is the symmetric
  group $\mathfrak{S}_4$, consists of the exceptional groups $G_{8}$, \ldots, $G_{15}$.
\item The remaining groups $G_{16}$, \ldots, $G_{22}$ form the
  \emph{icosahedral} family, where $W/Z(W)$ is the alternating group
  $\mathfrak{A}_5$.
\end{itemize}

As mentioned before, the BMM symmetrising trace conjecture is known to
be true for the tetrahedral family.  In this paper we introduce a new
strategy to verify this conjecture for the octahedral family, and to
recover the results for the tetrahedral family, having as a goal to
introduce a uniform method applying to all exceptional groups of rank
2, including the icosahedral family.

For this purpose, we take advantage of the following facts: firstly, each family has a maximal group $W_{\max}$  (these are the groups $G_7$,$G_{11}$, and $G_{19}$ for each family, respectively), and each group $W$ in the family is a normal reflection
subgroup of $W_{\max}$. Moreover, if $Z(W_{\max})=\Span{z}$, then $Z(W)=\Span{z^{\ell}}$, where $\ell$ is the index of $W$ in $W_{\max}$. 

We now consider the groups of the tetrahedral and octahedral family.
For each group in the family (including the maximal group), we use a factorization 
$$
	W= Z(W) W' X= \{yvx\,:\,y\in Z(W),\,v\in W',\,x\in X\}
$$
of
$W$, where $W'$ is a parabolic subgroup of $W$ (with $W' \cap Z(W) = 1$),
and $Z(W)X$ is a set of representatives of the right cosets of $W'$ in $W$. The set $X$ is the vertex set of a spanning tree with edges labelled by generators of~$W$. 

 In the tetrahedral family, we use the same parabolic subgroup $W'$ for all groups and we choose a set $X$ for $G_6$, then we lift to obtain such a set for $G_7$, from where we restrict it to obtain such sets for $G_4$ and $G_5$.
 
 In the octahedral family we use the same parabolic subgroup for the groups $G_8$, $G_9$, $G_{10}$, and $G_{11}$ and we choose a set $X$ for $G_9$, then we lift to obtain such a set for $G_{11}$, from where we restrict it to obtain such sets for $G_{8}$ and $G_{10}$. For the rest of the groups in this family, we choose a factorization independently.
 
 Now, as far as the Hecke algebras are
concerned,  we use the aforementioned  factorization as an inspiration to construct a $z$-basis $\mathcal{B}$ for each $H(W)$. In order to prove that this is indeed a basis, we use a \emph{coset table}, similar to the one we have used in \cite{ChP}, which allows us to work with a faithful matrix representation of $H(W)$,
rather than its presentation. Hence, we are able to
automate calculations inside the Hecke algebra. This particular construction of  $z$-basis and the automated calculations allow us to verify the validity of the BMM symmetrising trace conjecture, with respect to $\tau_{\mathcal{B}}$. 

We are optimistic that one could  find suitable $z$-bases that would allow for a similar strategy to work in
the icosahedral family.

Programs with computer calculations for our result can be found at the project's webpage \cite{pr}.
\\

\noindent
\textbf{Acknowledgments.} Research supported by the Hellenic Foundation for Research and Innovation (H.F.R.I.) under the Basic Research Financing (Horizontal support for all Sciences), National Recovery and Resilience Plan (Greece 2.0), Project Number: 15659, Project Acronym: SYMATRAL.

\section{Preliminaries}
\subsection{Complex reflection groups}
A \emph{pseudo-reflection} is a non-trivial
element of $\mathrm{GL}_n(\mathbb{C})$ of finite order,  whose fixed points in $\mathbb{C}^n$ form a hyperplane. A pseudo-reflection $s$ is called \emph{distinguished} if its only non-trivial eigenvalue on $\mathbb{C}^n$ equals $\exp(-2\pi \sqrt{-1}/e_s)$, where  $e_s$ denotes the order of $s$.

A complex reflection group $W$  is a finite subgroup of $\mathrm{GL}_n(\mathbb{C})$ generated by \emph{pseudo-reflections}. By a change of basis one
can always assume that $W$ is a finite subgroup of $\mathrm{GL}_n(K)$, where $K$ denotes the \emph{field of definition} of $W$, that is a subfield of $\mathbb{C}$ generated by the traces  of all the elements of $W$.

 Real reflection groups, also known as \emph{finite Coxeter groups} are particular cases of complex reflection groups. Therefore, complex reflection groups include interesting examples of groups, such as the Weyl groups and the dihedral groups.

A complex reflection group $W$ is \emph{irreducible} if it acts irreducibly on $\mathbb{C}^n$. In this case, we call $n$ the  \emph{rank} of $W$. Irreducible complex reflection groups are very important, since
every complex reflection group can be written as a  direct product of irreducible ones. The classification of irreducible complex reflection groups is thanks to  Shephard and Todd \cite{ShTo} and it is given by the following theorem:

\begin{Theorem}\label{ShToClas} Let $W \subset \mathrm{GL}_n(\mathbb{C})$ be an irreducible complex
	reflection group. Then, up to conjugacy, $W$ belongs to precisely one of the following classes:
	\begin{itemize}
		\item The symmetric group $S_{n+1}$.
		\item The infinite family $G(de,e,n)$, where  $d,e,n \in \mathbb{N}^*$, such that  $(de,e,n)\neq (1,1,n)$ and $(de,e,n)\neq(2,2,2)$. In this family we encounter all
		$n\times n$ monomial matrices whose non-zero entries are ${de}$-th roots of unity, while the product of all non-zero
		entries is a $d$-th root of unity.
		\smallbreak
		\item The 34 exceptional groups labelled
		$G_4, \, G_5, \, \dots,\, G_{37}$ (ordered with respect to increasing rank).
	\end{itemize}
\end{Theorem}

 Among the exceptional groups we encounter the \emph{exceptional groups of rank 2}, which are the groups $G_4,\dots, G_{22}$. It is known that the quotient of each of these groups by its centre is either
 the tetrahedral, octahedral or icosahedral group (for more details one may refer to Chapter 6 of \cite{lehrer}). Therefore, the exceptional groups of rank 2 are divided into three families: the first family, known as \emph{the tetrahedral family}, includes the groups $G_4,\dots, G_7$, the second one, known as the \emph{octahedral family} includes the groups $G_8,\dots, G_{15}$ and the last one, known as the \emph{icosahedral family}, includes the rest of them. In each family, there is a maximal group of order $|W/Z(W)|^2$ (these are the groups $G_7$, $G_{11}$ and $G_{19}$, respectively for each family) and the rest of the groups in this family can be seen as subgroups of the maximal group. In Table \ref{t1} one can see a presentation of the maximal group for each family.
 \begin{table}[!ht]
 	\begin{center}
 		\small
 		 \caption{
 			\bf{The exceptional groups of rank $2$}}
 		\label{t1}
 		\scalebox{0.9}
 		{\begin{tabular}{|c|c|c|}
 				\hline
 				$W/Z(W)$& $W$& Maximal Group \\
 				\hline
 				$\begin{array}[t]{lcl} \\\text{Tetrahedral Group }
 					\mathcal{T}\simeq \mathfrak{A}_4\\ \\
 					\langle s,t,u\;|\;s^2=t^3=u^3=1, stu=1\rangle\end{array}$
 				&
 				$\begin{array}[t]{lcl}\\G_4,\dots,G_7\end{array}$
 				&
 				$\begin{array}[t]{lcl}\\ G_7=
 					\langle s,t,u\;|\;s^2=t^3=u^3=1, stu \text{ is central }\rangle
 				\end{array}$\\
 				\hline \hline

 				$\begin{array}[t]{lcl} \\\text{Octahedral Group }
 					\mathcal{O}\simeq \mathfrak{S}_4\\\\
 					\langle s,t,u\;|\;s^2=t^3=u^4=1, stu=1\rangle\end{array}$
 				&$\begin{array}[t]{lcl}\\
 					G_8,\dots,G_{15}
 				\end{array}$&
 				$\begin{array}[t]{lcl}\\ G_{11}=
 					\langle s,t,u\;|\;s^2=t^3=u^4=1, stu \text{ is central }\rangle
 				\end{array}$\\
 				\hline \hline

 				$\begin{array}[t]{lcl} \\\text{Icosahedral Group }
 					\mathcal{I}\simeq \mathfrak{A}_5\\\\
 					\langle s,t,u\;|\;s^2=t^3=u^5=1, stu=1\rangle\end{array}$
 				&$\begin{array}[t]{lcl}\\
 					G_{16},\dots,G_{22}
 				\end{array}$&
 				$\begin{array}[t]{lcl}\\ G_{19}=
 					\langle s,t,u\;|\;s^2=t^3=u^5=1, stu \text{ is central }\rangle
 				\end{array}$\\
 				\hline
 		\end{tabular}}
 	\end{center} 
 	\end{table}

\subsection{Hecke algebras}
Let $W$ be a complex reflection group. Brou\'e, Malle and Rouquier \cite[\S 2 B]{BMR} associated to $W$ a \emph{complex braid group} $B(W)$. Complex braid groups generalize the definition of \emph{Artin-Tits groups} \cite{Tits}, which were classically associated to real reflection groups.

Let $s\in W$ be a pseudo-reflection. Brou\'e, Malle and Rouquier
\cite[\S 2 B]{BMR} associated to $s$ an element $\mathbf{s}\in B(W)$,
which they called \emph{braided reflection}.
% The element $\mathbf{s}$ is
% unique, up to conjugacy (see \cite[Lemma 2.12 (2)]{Br}).
The
following result is \cite[Theorem 0.1]{be}:

\begin{Theorem}\label{pre}
	Let $W$ be a complex reflection group with associated complex braid group $B(W)$. There exists a finite subset $\mathbf{S}=\{\mathbf{s}_1,\dots, \mathbf{s}_n\}$ of $B(W)$, such that:
	\begin{itemize}
		\item[(i)] The elements $\mathbf{s}_1, \dots, \mathbf{s}_n$ are braided reflections associated to distinguished pseudo-reflections $s_1,\dots, s_n$ in $W$.
		\item[(ii)]The set  $\mathbf{S}$ generates $B(W)$ and the set $S:=\{s_1,\dots, s_n\}$ generates $W$.
		\item [(iii)] There exists a set $I$ of \emph{braid relations}, which are of the form $\mathbf{w_1}=\mathbf{w_2}$, where $\mathbf{w_1}$ and $\mathbf{w_2}$ are positive words of equal length in the elements of $\mathbf{S}$, such that $\langle \mathbf{S}\;|\;I\rangle$ is a presentation of $B(W)$.
		\item[(iv)] Viewing now $I$ as a set of relations in $S$, the group
		$W$ is presented by $$\langle S \;|\;I\;; \forall s\in S ,\, s^{e_s}=1\rangle$$.
	\end{itemize}
\end{Theorem}

The previous theorem allows us to obtain elements of $W$ from elements
of $B(W)$. More precisely, let $\mathbf{w}\in B(W)$. According to the
previous theorem, $\mathbf{w}$ is a word in the braided reflections
$\{\mathbf{s}_1,\dots, \mathbf{s}_n\}$, which generate $B(W)$. We
denote by $w \in W$ the group element obtained from $\mathbf{w}$
by replacing each $\mathbf{s}_i$ with $s_i$.  This process is independent of
the chosen representation of $\mathbf{w}$ as word in the generators of $B(W)$, since the map $\mathbf{w} \mapsto w$ is a homomorphism.

Let $\mathcal{R}$ denotes the set of all distinguished
pseudo-reflections of $W$.  For each $s\in \mathcal{R}$ we choose a
set of $e_s$ indeterminates $u_{s,1},\dots, u_{s,e_s}$, such that
$u_{s,j}=u_{t,j}$ if $s, t \in \mathcal{R}$ are conjugate in $W$. We
denote by $R(W)$ the Laurent polynomial ring
$\mathbb{Z}[u_{s,j},u_{s,j}^{-1}]$.

The \emph{generic Hecke algebra} $H(W)$ associated to $W$ is the
quotient of the group algebra $R(W)[B(W)]$ of $B(W)$ by the ideal
generated by the elements of the form
\begin{equation}
	(\mathbf{s}-u_{s,1})(\mathbf{s}-u_{s,2})\cdots (\mathbf{s}-u_{s,e_s}),
	\label{Hecker}
\end{equation}
where $s$ runs over a set of representatives of the $W$-conjugacy classes of
$\mathcal{R}$. % $\mathbf{s}$ \emph{braided reflections} associated to the pseudo-reflection $s$

Let $\theta : R(W)\rightarrow K$ the natural specialization
$u_{s,j}\mapsto {\rm exp}(2\pi \sqrt{-1}\; j/e_s)$ for every $s\in \mathcal{R}$ and $j=1,\ldots, e_s$. Since, $H(W)\otimes_{\theta}K\simeq K[W]$, the generic Hecke algebra $H(W)$ can be seen as a
deformation of the group algebra $K[W]$.

If $W$ is a real reflection group, $H(W)$ is known as the \emph{Iwahori--Hecke algebra} associated with $W$ (for more details about Iwahori--Hecke algebras one may refer, for example, to \cite[\S 4.4]{GP}).

We obtain an equivalent definition of $H(W)$ if we expand the relations \eqref{Hecker}.
More precisely, $H(W)$ is the quotient of the group algebra $R(W)[B(W)]$  by the elements of the form
\begin{equation}
	{\mathbf{s} }^{e_{s}}-a_{{s},e_{s}-1}{\mathbf{s}}^{e_{s}-1}-a_{{s},e_{s}-2}{\mathbf{s}}^{e_{s}-2}-\dots-a_{{s},0},
	\label{Hecker2}
\end{equation}
where $a_{{s},e_{s}-k}:=(-1)^{k-1}f_k({u}_{{s},1},\dots,{u}_{{s},e_{s}})$ with $f_k$ denoting the $k$-th elementary symmetric polynomial, for $k=1,\ldots,e_{s}$.

Inspired by the notation of the real case, we denote by $b\mapsto T_{b}$ the restriction of the natural surjection $R(W)[B(W)] \rightarrow H(W)$ to $B(W)$.
In the presentation of $H(W)$, apart from
the braid relations coming from the presentation of $B(W)$ (see Theorem \ref{pre} (iii)), we also have the \emph{positive Hecke relations}:
\begin{equation}
	T_{\mathbf{s}}^{e_{s}}=a_{{s},e_{s}-1}T_{\mathbf{s}}^{e_{s}-1}+a_{{s},e_{s}-2}T_{\mathbf{s}}^{e_{s}-2}+\dots+a_{{s},0}.
	\label{ph}
\end{equation}
If we multiply \eqref{ph} with $T_{\mathbf{s}^{-1}}$ we can see that $T_{\mathbf{s}}$ is invertible in $H(W)$ with
\begin{equation}\label{invhecke}
	T_{\mathbf{s} }^{-1}=a_{s,0}^{-1}\,T_{\mathbf{s} }^{e_{s}-1}-a_{s,0}^{-1}\,a_{{s},e_{s}-1}T_{\mathbf{s} }^{e_{s}-2}-a_{s,0}^{-1}\,a_{{s},e_{s}-2}T_{\bf s}^{e_{s}-3}-\dots- a_{s,0}^{-1}\,a_{{s},1}.
\end{equation}
We call  relations \eqref{invhecke} the \emph{inverse Hecke relations}.

\subsection{The two fundamental conjectures}
Iwahori--Hecke algebras have a \emph{standard basis} denoted by $(T_w)_{w \in W}$, indexed by the elements of the finite Coxeter group $W$ (see \cite[IV, \S 2]{Bou05}). Therefore, for the real case, $H(W)$ is a free $R(W)$-module of rank $|W|$. A crucial role to the existence of the standard basis is played by the \emph{length function} on $W$. Such a function does not exist for the complex reflection groups in general and, therefore, one could not repeat the proof of real case to include the complex case, as well.
	Brou\'e, Malle and Rouquier conjectured such a result for non-real complex reflection groups \cite[\S 4]{BMR}:

	\begin{conj}[The BMR freeness conjecture]\label{BMR free}
		The algebra $H(W)$ is a free $R(W)$-module of rank  $|W|$.
	\end{conj}

This conjecture has been solved since 2017, after almost 20 years since it was first formulated. The proof uses
a case-by-case analysis approach. It has turned out to be a collective outcome of several people that have provided with their own results through their papers over the course of the past years. A detailed state of the art of this proof can be found in \cite[Theorem 3.5]{BCCK}.

The proof of the BMR freeness conjecture includes
the construction of an $R(W)$-basis consisting of $|W|$ elements for $H(W)$. More precisely, the first author \cite{Ch17, Ch18} has
constructed explicit bases for the cases of $G_4,\ldots,G_{15}$, which are of a particular form. This form has been used in \cite{BCC,BCCK, CC, ChP} and we use it also for the purposes of this paper. Therefore, it is necessary to give a detailed description of it.

Suppose that $W$ is an irreducible complex reflection group.
It is known (\cite[Theorem 12.8]{bee} and \cite[Theorem 2.24]{BMR}) that $Z(B(W))$ is infinite cyclic, generated by $\mathbf{z}$,  and that
$z$ is generator of the cyclic group $Z(W)$ of order $k$.
Let $m:=|W/Z(W)|$ and let $s$ be a distinguished pseudo-reflection of order $e_s$.

\begin{defi}A basis $\mathcal{B}$ of $H(W)$ as $R(W)$-module is called $z$-basis, if there are braid groups elements $\mathbf{b}_1,\dots, \mathbf{b}_{m/e_s}\in B(W)$ with  $\mathbf{b}_1=1$, such that
	$$\mathcal{B}=\{1_{H(W)},T_{\mathbf{z}},\dots, T_{\mathbf{z}}^{k-1}\}\cdot \{1_{H(W)},T_{\mathbf{s}}, \dots, T_{\mathbf{s}}^{e_s-1}\}\cdot \{T_{\mathbf{b}_1},\dots, T_{\mathbf{b}_{m/e_s}}\}$$
	\end{defi}
Note that the set $\{1_{H(W)},\mathbf{s}, \dots, \mathbf{s}^{e_s-1}\}$ in this definition corresponds to a rank $1$ parabolic subgroup $W'$ of $W$, and that $\{1_{H(W)},T_{\mathbf{z}},\dots, T_{\mathbf{z}}^{k-1}\}\cdot \{T_{\mathbf{b}_1},\dots, T_{\mathbf{b}_{m/e_s}}\}$ corresponds to a set of coset representatives of $W/W'$.

In \cite{Ch18} the first author provides $z$-bases for the cases of $G_4,\dots, G_{15}$ and in \cite{ChP} we provide such a basis also for $G_{22}$.

\begin{ex}
	Let $G_6 = \Span{s, t \mid s^2 = t^3 = 1, ststst = tststs}$. We have $Z(G_6)=\Span{z}$, where $z=ststst$ with $z^4=1$.
In order to make notation lighter, we denote again with $s$ and $t$ the generators of $H(G_6)$ (instead of $T_{\mathbf{s}}$ and $T_{\mathbf{t}}$) and with $z$ the element $T_{\mathbf{z}}$.
We know from \cite{Ch18} that $H(G_6)$ admits the following basis $\mathcal{B}$, as $R(G_6)$ module:
	$$\left\{
	\begin{matrix}
		z^k, z^k{t}, z^k{t}^2, z^k{t}, z^k{s}{t}, z^k{s}{t}^2, z^k{ts}, z^k{t}^2{s}, z^k{tst}, z^k{ts}{t}^2, z^k{t}^2{st}, z^k{t}^2{s}{t}^2
	\end{matrix}\left|\,\,k=0,1,2,3\right.\right\}.
	$$
 One can easily see that this is a $z$-basis, since $\mathcal{B}=\{1,z,z^2,z^3\}\cdot\{1, {t},{t}^2\}\cdot\{1,{s},{st},{s}{t}^2\}$.
	\qed
\end{ex}

Revisiting the real case again,  Iwahori-Hecke algebras $H(W)$ are examples of symmetric algebras
 over their ring of definition $R(W)$, namely they admit a linear map $\tau:  H(W) \rightarrow R(W)$, such that
 $\tau(hh') = \tau(h'h)$ for all $h, h' \in H(W)$ and the bilinear form
 $(h, h') \mapsto \tau(hh')$ is non-degenerate. Such a linear map is called \emph{symmetrising trace}.

 A symmetrising trace for Iwahori-Hecke algebras is given by $\tau(T_w)=\delta_{1w}$  \cite[IV, \S 2]{Bou05} and it turns out that it is unique. Brou\'e, Malle and Michel  conjectured the existence of a unique symmetrising trace also for non-real complex reflection groups \cite[\S2.1, Assumption 2(1)]{BMM}:

	\begin{conj}[The BMM symmetrising trace conjecture]
		\label{BMM sym}
		Let $W$ be a complex reflection group.
		There exists a linear map $\tau: H(W)\rightarrow R(W)$ such that:
		\begin{itemize}
			\item[$(1)$] $\tau$ is a symmetrising trace, that is, we have $\tau(h_1h_2)=\tau(h_2h_1)$ for all  $h_1,h_2\in H(W)$, and the bilinear map $H(W)\times H(W)\rightarrow R(W)$,  $(h_1,h_2)\mapsto\tau(h_1h_2)$ is non-degenerate. \smallbreak
			\item[$(2)$] $\tau$ becomes the canonical symmetrising trace on $K[W]$ when ${u}_{s,j}$ specialises to ${\rm exp}(2\pi \sqrt{-1}\; j/e_s)$ for every $s\in S$ and $j=1,\ldots, e_s$. \smallbreak
			\item[$(3)$]  $\tau$ satisfies
			\begin{equation}\label{extra}
				\tau(T_{b^{-1}})^* =\frac{\tau(T_{b\boldsymbol{\pi}})}{\tau(T_{\boldsymbol{\pi}})}, \quad \text{ for all } b \in B(W),
			\end{equation}
			where
			%$b\mapsto T_{b}$ denotes the restriction of the natural surjection $R(W)[B(W)] \rightarrow H(W)$ to $B(W)$,
			$x \mapsto x^*$ is the automorphism of  $R(W)$ given by $u_{s,j} \mapsto u_{s,j}^{-1}$ and $\boldsymbol{\pi}$ the element $\mathbf{z}^{|Z(W)|}$.
			%, with $z$ being again the image of a suitable generator of the centre of $B(W)$ inside $H(W)$.
			\smallbreak
		\end{itemize}
	\end{conj}
	Since we have the validity of the BMR freeness conjecture, we know  \cite[\S 2.1]{BMM} that if there exists such a linear map $\tau$, then it is unique. If this is the case, we call $\tau$ the \emph{canonical symmetrising trace on}  $H(W)$.

Apart from the real case, the BMM symmetrising trace conjecture is known to hold  for a few exceptional groups and for the infinite family (detailed references can be found in \cite[Conjecture 3.3]{CC}).

Malle and Michel suggested a construction of a basis by lifting the
elements of $W$ to $B(W)$, which satisfy an additional property with
respect to $\tau$ \cite[Conjecture 2.6]{MM10}:

\begin{conj}[The lifting conjecture]\label{lift}
  There exists a section (that is, a map admitting a left inverse)
  $W \rightarrow \boldsymbol{W}\subset B(W)$,
  $w \mapsto \boldsymbol{w}$ of $W$ in $B(W)$ such that
  $1 \in \boldsymbol{W}$, and such that for any
  $\boldsymbol{w} \in \boldsymbol{W}$ we have
  $\tau(T_{\boldsymbol{w}}) = \delta_{1\boldsymbol{w}}$.
\end{conj}

If the lifting conjecture holds, then Condition (2) of Conjecture
\ref{BMM sym} is obviously satisfied. If further the elements
$\{T_{\boldsymbol{w}}\,|\,\boldsymbol{w} \in \boldsymbol{W}\}$ form an
$R(W)$-basis of $H(W)$, then, by \cite[Proposition 2.7]{MM10},
Condition \eqref{extra} of Conjecture \ref{BMM sym} is equivalent to:
\begin{equation}\label{extra2}
  \tau(T_{x^{-1}\boldsymbol{\pi}})=0, \quad \text{ for all } x \in \mathcal{B} \setminus \{1_{H(W)}\}.
\end{equation}

\begin{rem}\label{remex}
  If $H(W)$ admits a $z$-basis, then we have the validity of the
  lifting conjecture \ref{lift}.
\end{rem}

Since generic Hecke algebras are free algebras (we have the validity
of the BMR freeness conjecture \ref{BMR free}), a linear map $\tau$ is
a symmetrising trace, if the Gram matrix
$A:=(\tau(bb'))_{b,b'\in \mathcal{B}}$ is symmetric and its
determinant is a unit in $R(W)$ for some (and hence every) basis
of the algebra.

Using this approach, the proof of the
BMM symmetrising trace conjecture depends on the choice of a suitable
basis $\mathcal{B}$ for the generic Hecke algebra.
If $1_{H(W)} \in \mathcal{B}$ then we can define a linear map
$\tau_{\mathcal{B}}\colon H(W) \rightarrow R(W)$ by setting
$\tau_{\mathcal{B}}(1_{H(W)}) = 1$ and $\tau_{\mathcal{B}}(b) = 0$ for
$b \neq 1_{H(W)}$.

\section{Towards the validity of the BMM symmetrising trace conjecture}

The goal of this section is to find a suitable $z$-basis $\mathcal{B}$ for the generic Hecke algebra associated to each of the groups of the tetrahedral and octahedral family, such that $\tau_{\mathcal{B}}$  satisfies the BMM symmetrising trace conjecture.
These are the groups $G_j$, $j=4,\dots,15$ in the Shephard-Todd classification.

Briefly, the procedure is the following: For each complex reflection group $G_j$ under consideration, we use a factorization 
\begin{align*}
	G_j = Z_j G_j' X_j = \{yvx\,:\,y\in Z_j,\,v\in G_j',\,x\in X_j\}
\end{align*}
of
$G_j$, where $Z_j$ is the centre of $G_j$, $G_j'$ is a parabolic subgroup of $G_j$ (with $G_j' \cap Z_j = 1$),
and $Z_jX_j$ is a set of representatives of the right cosets of $G_j'$ in $G_j$.

As in Section~2.2 of~\cite{ChP}, the set $X_j$ is the vertex set of a spanning tree with edges labelled by generators of~$G_j$.  For each $x \in X_j$, this tree defines a unique word in the generators whose product is~$x$.

Considering the centre $Z_j$ of $G_j$, it is known that $Z_j$ is a subgroup of the centre of the maximal group in the family.
In the tetrahedral family, i.e., for $j=4,\dots, 7$, let $z$ be a generator of the centre of $G_7$, which is the maximal group of this family, and set $z_j:=z^{\ell_j}$, where $\ell_j = |G_7 : G_j|$.  Then,
for each $G_j$, $j=4,\dots, 7$, we have
$Z_j = G_j \cap Z_7$ and, hence,  $Z_j=\Span{z_j}$.
 There is an analogue for the groups in the octahedral family, i.e., for $j=8,\dots,15$, replacing $G_7$ by the maximal group $G_{11}$ of this family.

Our goal is to use the aforementioned  factorization as an inspiration to construct a $z_j$-basis for each $H(G_j)$.
The rest of this section is devoted to this procedure.

\subsection{Factorization}\label{sec:factorization}
\subsubsection{Tetrahedral family}\label{tet}
The tetrahedral family consists of the complex reflection groups
$G_4$, $G_5$, $G_6$ and $G_7$.  Each of them is a normal reflection
subgroup of $G_7$, as illustrated in Figure~\ref{fig:tetrahedral}. In
this figure, there appear also the groups $G(4,4,2)$ of the infinite
family and the quaternion group $Q_8$, in order to complete the
diagram.

%The goal of this section is to find a suitable $z$-basis $\mathcal{B}$ for the generic Hecke algebra associated to each of the groups of the tetrahedral family, such that $\tau_{\mathcal{B}}$  satisfies the BMM symmetrizing trace conjecture.The rest of this section is devoted to this procedure.

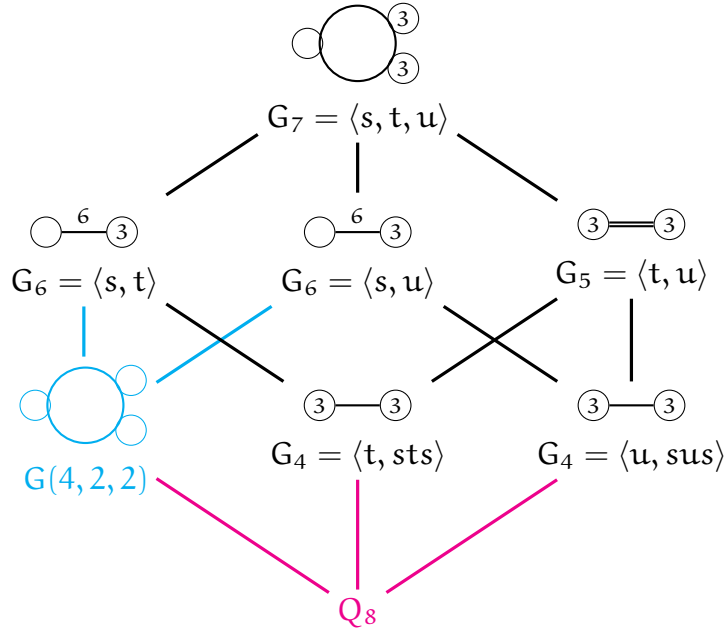
\begin{figure}[htb]
	\begin{center}
		\begin{tikzpicture}[xscale=1.8,yscale=1.2]
			\node[rectangle,inner sep=1mm] (G7) at (0,10) {
				\begin{tikzpicture}%[baseline=(1.text)]
					\draw[thick] (0,0) circle (5mm);
					\node[draw,circle] (1) at (-0.66,0) {};
					\node[draw,circle,inner sep=0.5mm] (2) at (0.6,0.33) {$\scriptstyle 3$};
					\node[draw,circle,inner sep=0.5mm] (3) at (0.6,-0.33) {$\scriptstyle 3$};
					\node[rectangle] at (0,-1) {$G_{7} = \Span{s,t,u}$};
				\end{tikzpicture}
			};
			\node[rectangle,inner sep=1mm] (G6a) at (-2,8) {
				\begin{tikzpicture}[baseline=(2.text)]
					\node[draw,circle] (1) at (0,0) {};
					\node[draw,circle,inner sep=0.5mm] (2) at (1,0) {$\scriptstyle 3$};
					\draw[thick] (1) -- node[above] {$_6$} (2);
					\node[rectangle] at (0.5,-0.67) {$G_{6} = \Span{s,t}$};
				\end{tikzpicture}
			};
			\node[rectangle,inner sep=1mm] (G6b) at (0,8) {
				\begin{tikzpicture}[baseline=(2.text)]
					\node[draw,circle] (1) at (0,0) {};
					\node[draw,circle,inner sep=0.5mm] (2) at (1,0) {$\scriptstyle 3$};
					\draw[thick] (1) -- node[above] {$_6$} (2);
					\node[rectangle] at (0.5,-0.67) {$G_{6} = \Span{s,u}$};
				\end{tikzpicture}
			};
			\node[rectangle,inner sep=1mm] (G5) at (2,8) {
				\begin{tikzpicture}[baseline=(1.text)]
					\node[draw,circle,inner sep=0.5mm] (1) at (0,0) {$\scriptstyle 3$};
					\node[draw,circle,inner sep=0.5mm] (2) at (1,0) {$\scriptstyle 3$};
					\draw[thick,double] (1) -- (2);
					\node[rectangle] at (0.5,-0.67) {$G_{5} = \Span{t,u}$};
				\end{tikzpicture}
			};
			\node[cyan,rectangle,inner sep=1mm] (G422) at (-2,6) {
				\begin{tikzpicture}[baseline=(1.text)]
					\draw[thick] (0,0) circle (5mm);
					\node[draw,circle] (1) at (-0.66,0) {};
					\node[draw,circle] (2) at (0.6,0.33) {};
					\node[draw,circle] (3) at (0.6,-0.33) {};
					\node[rectangle] at (0,-1) {$G(4,2,2)$};
				\end{tikzpicture}
			};
			\node[rectangle,inner sep=1mm] (G4a) at (0,6) {

				\begin{tikzpicture}[baseline=(1.text)]
					\node[draw,circle,inner sep=0.5mm] (1) at (0,0) {$\scriptstyle 3$};
					\node[draw,circle,inner sep=0.5mm] (2) at (1,0) {$\scriptstyle 3$};
					\draw[thick] (1) -- (2);
					\node[rectangle] at (0.5,-0.67) {$G_{4} = \Span{t,sts}$};
				\end{tikzpicture}
			};
			\node[rectangle,inner sep=1mm] (G4b) at (2,6) {
				\begin{tikzpicture}[baseline=(1.text)]
					\node[draw,circle,inner sep=0.5mm] (1) at (0,0) {$\scriptstyle 3$};
					\node[draw,circle,inner sep=0.5mm] (2) at (1,0) {$\scriptstyle 3$};
					\draw[thick] (1) -- (2);
					\node[rectangle] at (0.5,-0.67) {$G_{4} = \Span{u,sus}$};
				\end{tikzpicture}
			};
			\node[magenta,rectangle,inner sep=1mm] (Q8) at (0,4) {$\mathop{Q_8}%\limits_{[8]}
				$};

			\draw[very thick] (G6b) -- (G7);
			\draw[very thick] (G5) -- (G7);
			\draw[very thick] (G6a) -- (G7);
			\draw[very thick,cyan] (G422) -- (G6a);
			\draw[very thick,cyan] (G422) -- (G6b);
			\draw[very thick] (G4a) -- (G6a);
			\draw[very thick] (G4a) -- (G5);
			\draw[very thick] (G4b) --  (G6b);
			\draw[very thick] (G4b) --  (G5);
			\draw[very thick,magenta] (Q8) --  (G422);
			\draw[very thick,magenta] (Q8) --  (G4a);
			\draw[very thick,magenta] (Q8) --  (G4b);
		\end{tikzpicture}
	\end{center} 
	\caption{Tetrahedral Family.}
	\label{fig:tetrahedral}
\end{figure}

%%%%%%%%%%%%%%%%%%%%%%%%%%%%%%%%%%%%%%%%%%%%%%%%%%%%%%%%%%%%%%%%%%%%%%%%%%%%%

The maximal group $G_7$ of the tetrahedral family has the
presentation
\begin{align*}
	G_7 = \Span{s, t, u \mid s^2 = t^3 = u^3 = 1,\,
		stu = tus = ust}.
\end{align*}
We set $z: = stu$, a generator of the centre of $G_7$.  Thus
\begin{align*}
	Z_7: = \Span{z} = \{1, z, z^2, \dots, z^{11}\}
\end{align*}
is the centre of $G_7$.

We now take advantage of the fact that
representatives for each of the groups $G_j$, $j = 4,5,6,7$,
can be chosen as subgroups of $G_7$ in such a way that they contain
the subgroup $U = \Span{u} = \{1, u, u^2\}$ as a parabolic subgroup.
So we have $G_j=Z_jUX_j=\{yvx\,:\,y\in Z_j,\,v\in U,\,x\in X_j\}$, for
$j=4,5,6,7$, where $Z_j$ is the centre of $G_j$.
% Since here $Z_j = G_j \cap Z_7$, we have $Z_j=\Span{z^{\ell_j}}$, where $\ell_j = |G_7 : G_j|$.
%
It remains to choose $X_j$ so that $Z_j X_j$ is a set of (right) coset representatives of $U$ in $G_j$.

We start by choosing a set $X_6$, then we lift it to  obtain $X_7$, from where we restrict it to obtain $X_5$ and $X_4$.
The details are as follows.
\smallbreak\smallbreak \smallbreak  \noindent
\textbf{The group} $\mathbf{G_6.}$ $$G_6 = \Span{s, u} \leq G_7.$$
Using $t^3 = 1$, it follows that
\begin{align*}
	z^3 = (stu)^3 = su su su = us us us.
\end{align*}
Note that $Z_6= \Span{z^3} = \{1, z^3, z^6, z^9\}$ is the centre of $G_6$.

We choose
\begin{align*}
	X_6 = \{1, s, su, sus\}.
\end{align*}
Then
\begin{align*}
	G_6 = Z_6 \,  U \,  X_6 = \{y v x : y \in Z_6, \, v \in U, \, x \in X_6\}
\end{align*}
$X_6$ is a vertex set of a spanning tree, as illustrated in the following figure:
\begin{figure}[htb]
	\begin{center}
		\begin{tikzpicture}
			\node (1) at (0,0) {$x_1$};
			\node (2) at (1,0) {$x_2$};
			\node (3) at (2,0) {$x_3$};
			\node (4) at (3,0) {$x_4$};

			\draw (1) -- node[above] {$s$} (2);
			\draw (2) -- node[above] {$u$} (3);
			\draw (3) -- node[above] {$s$} (4);
		\end{tikzpicture}
	\end{center} 
\end{figure}
\smallbreak\smallbreak \smallbreak \noindent
\textbf{The group} $\mathbf{G_7.}$
 $$G_7 = \Span{s, t, u}.$$
The set $\{1, z, z^2\}$ is a transversal of
the cosets of $G_6$ inside $G_7$, i.e., $G_7 = G_6 \cup z G_6 \cup z^2 G_6$.
From $Z_7 = Z_6 \cup z Z_6 \cup z^2 Z_6$
it follows that
\begin{align*}
	G_7 = Z_7\,  U\,  X_7,
\end{align*}
where $X_7 = X_6 = \{1, s, su, sus\}$.
\smallbreak\smallbreak\smallbreak  \noindent
			\textbf{The group} $\mathbf{G_5.}$ $$G_5 = \Span{t, u}.$$
		Using $s^2 = 1$, it follows that
		\begin{align*}
			z^2 = (stu)^2 = tu tu = ut ut.
		\end{align*}
		Note that $Z_5 = \Span{z^2} = \{1, z^2, z^4, z^6, z^8, z^{10}\}$ is the centre of $G_5$.  It follows that
		\begin{align*}
			G_5 = Z_5\,  U\,  X_5,
		\end{align*}
		where
		\begin{align*}
			X_5 = (X_7 \cup z X_7) \cap G_5
			= \{1, sus, stu \cdot s, stu \cdot su\}
			= \{1, tut^{-1}, tu, tu^2\}.
		\end{align*}
	 $X_5$ is a vertex set of a spanning tree, if $tu$ and $t^{-1}$ are added as generators:
\begin{figure}[htb]
	\begin{center}
		\begin{tikzpicture}
			\node (1) at (0,0) {$x_1$};
			\node (4) at (3,0) {$x_4$};
			\node (2) at (2,0) {$x_2$};
			\node (3) at (1,0) {$x_3$};

			\draw (1) -- node[above] {$tu$} (3);
			\draw (3) -- node[above] {$t^{-1}$} (2);
			\draw (2) -- node[above] {$tu$} (4);
		\end{tikzpicture}
	\end{center} 
\end{figure}
	% \begin{figure}[htb]
	% 	\begin{center}
	% 		\begin{tikzpicture}
	% 			\node (1) at (0,0) {$x_1$};
	% 			\node (2) at (1,0) {\,\,$x_3$\,\,};
	% 			\node (3) at (2,0.5) {\,\,$x_2$};
	% 			\node (4) at (2,-0.5) {\,\,$x_4$};

	% 			\draw (1) -- node[above] {$tu$} (2);
	% 			\draw (2) -- node[above] {$t^{-1}$\,} (3);
	% 			\draw (2) -- node[below] {$s$} (4);

	% 		\end{tikzpicture}
	% 	\end{center} 
	% \end{figure}
		% (There is an alternative choice using $X_5' = \{1, t, tu, tu^2\}$.
		% Also, from $Z_7 = Z_5 \cup z Z_5$, it follows that $G_7 = Z_7\, U\, X_5'$.)

		\smallbreak\smallbreak\smallbreak  \noindent
\textbf{The group} $\mathbf{G_4.}$ $$G_4 = \Span{u, v},$$ where $v = sus = tut^{-1}$.	Set $w = z^3 s$. Using $s^2 = t^3 = 1$ it follows that $w = (stu)^3s = uvu = vuv$. Moreover, $s^{-1}w s = w$, $s^{-1} v s = u$ and
								\begin{align*}
									w^2 = uvuvuv = (stu)^6.
								\end{align*}
								Note that $Z_4 = \Span{z^6} = \Span{w^2}$ is the centre of $G_4$.
								It follows that
								\begin{align*}
									G_4 = Z_4\,  U\,  X_4,
								\end{align*}
								where
								\begin{align*}
									X_4 = (X_6 \cup z^3 X_6) \cap G_4 = \{1, sus, sususu \cdot s, sususu \cdot su\} = \{1, v, w, wu\}.
								\end{align*}
								% (How does this come from $G_5$??.
								% Also, from $Z_6 = Z_4 \cup z Z_4$, it follows that $G_6 = Z_6\, U\, X_4$.)

				$X_4$ is the vertex set of a spanning tree, if $w$ is added as a generator:
			\begin{figure}[htb]
				\begin{center}
					\begin{tikzpicture}
						\node (1) at (0,0) {$x_1$};
						\node (2) at (1,0.7) {$x_2$};
						\node (3) at (1,-0.7) {\,$x_3$};
						\node (4) at (2,-0.7) {\,\,$x_4$};

							\draw (1) -- node[above] {$v$} (2);
						\draw (1) -- node[below, yshift=-0.1cm, xshift=-0.1cm] {$w$} (3);
						\draw (3) -- node[above] {\,$u$} (4);

					\end{tikzpicture}
				\end{center}

			\end{figure}

\subsubsection{Octahedral family}
The octahedral family consists of the complex reflection groups $G_8$, $G_9$, $G_{10}$, $G_{11}$, $G_{12}$, $G_{13}$, $G_{14}$, and $G_{15}$. The maximal group in this family is $G_{11}$ and the rest of the groups can be seen as normal subgroups of it, as illustrated in Figure \ref{fig:octahedral}.

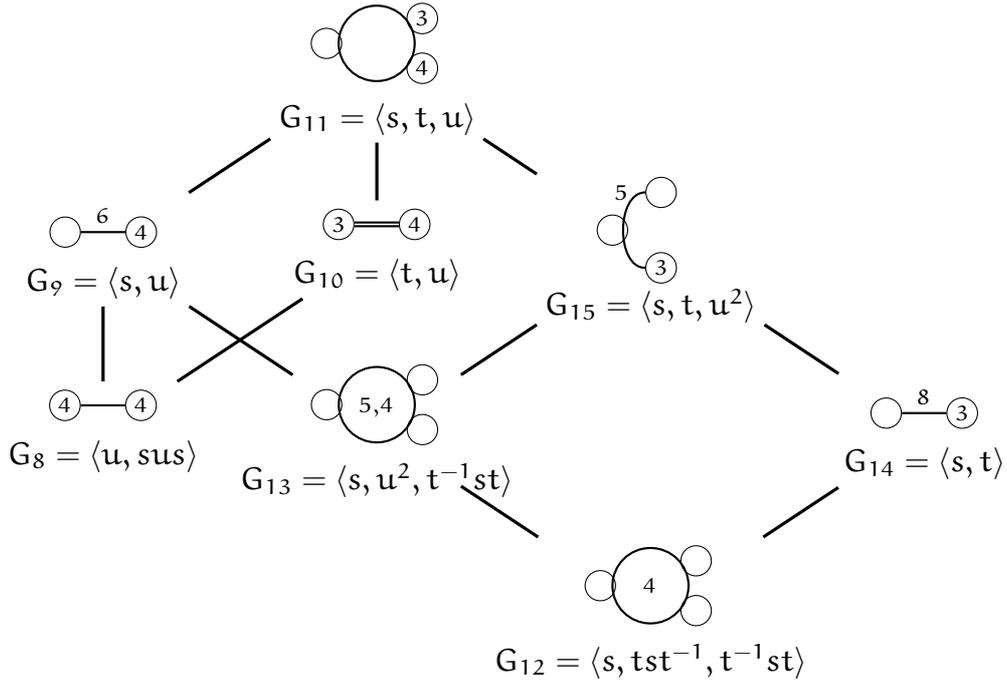
\begin{figure}[htb]
	\begin{center}
		\begin{tikzpicture}[xscale=1.8,yscale=1.2]
			\node[rectangle,inner sep=1mm] (G11) at (0,10) {
				\begin{tikzpicture}%[baseline=(1.text)]
					\draw[thick] (0,0) circle (5mm);
					\node[draw,circle] (1) at (-0.66,0) {};
					\node[draw,circle,inner sep=0.5mm] (2) at (0.6,0.33) {$\scriptstyle 3$};
					\node[draw,circle,inner sep=0.5mm] (3) at (0.6,-0.33) {$\scriptstyle 4$};
					\node[rectangle] at (0,-1) {$G_{11} = \Span{s,t,u}$};
				\end{tikzpicture}
			};
			\node[rectangle,inner sep=1mm] (G9) at (-2,8) {
				\begin{tikzpicture}[baseline=(2.text)]
					\node[draw,circle] (1) at (0,0) {};
					\node[draw,circle,inner sep=0.5mm] (2) at (1,0) {$\scriptstyle 4$};
					\draw[thick] (1) -- node[above] {$_6$} (2);
					\node[rectangle] at (0.5,-0.67) {$G_{9} = \Span{s,u}$};
				\end{tikzpicture}
			};
			\node[rectangle,inner sep=1mm] (G10) at (0,8) {
				\begin{tikzpicture}[baseline=(1.text)]
					\node[draw,circle,inner sep=0.5mm] (1) at (0,0) {$\scriptstyle 3$};
					\node[draw,circle,inner sep=0.5mm] (2) at (1,0) {$\scriptstyle 4$};
					\draw[thick,double] (1) -- (2);
					\node[rectangle] at (0.5,-0.67) {$G_{10} = \Span{t,u}$};
				\end{tikzpicture}
			};
			\node[rectangle,inner sep=1mm] (G15) at (2,8) {
				\begin{tikzpicture}[baseline=(t.text)]
					\node[draw,circle] (t) at (0,0) {};
					\node[draw,circle] (u) at (0.64,0.5) {};
					\node[draw,circle,inner sep=0.5mm] (1) at (0.64,-0.5) {$\scriptstyle 3$};
					\draw[thick] (u) node[left=0.3cm] {$_5$} to[out=180,in=180] (1);
					\node[rectangle] at (0.5,-1) {$G_{15} = \Span{s,t,u^2}$};
				\end{tikzpicture}
			};
			\node[rectangle,inner sep=1mm] (G8) at (-2,6) {
				\begin{tikzpicture}[baseline=(1.text)]
					\node[draw,circle,inner sep=0.5mm] (1) at (0,0) {$\scriptstyle 4$};
					\node[draw,circle,inner sep=0.5mm] (2) at (1,0) {$\scriptstyle 4$};
					\draw[thick] (1) -- (2);
					\node[rectangle] at (0.5,-0.67) {$G_{8} = \Span{u, sus}$};
				\end{tikzpicture}
			};
			\node[rectangle,inner sep=1mm] (G13) at (0,6) {
				\begin{tikzpicture}[baseline=(1.text)]
					\draw[thick] (0,0) circle (5mm) node (0) {$_{5,4}$};
					\node[draw,circle] (1) at (-0.66,0) {};
					\node[draw,circle] (2) at (0.6,0.33) {};
					\node[draw,circle] (3) at (0.6,-0.33) {};
					\node[rectangle] at (0,-1) {\hbox to 20mm{\hss$G_{13} = \Span{s,u^2,t^{-1}st}$\hss}};
				\end{tikzpicture}
			};
			\node[rectangle,inner sep=1mm] (G14) at (4,6) {
				\begin{tikzpicture}[baseline=(2.text)]
					\node[draw,circle] (1) at (0,0) {};
					\node[draw,circle,inner sep=0.5mm] (2) at (1,0) {$\scriptstyle 3$};
					\draw[thick] (1) -- node[above] {$_8$} (2);
					\node[rectangle] at (0.5,-0.67) {$G_{14} = \Span{s, t}$};
				\end{tikzpicture}
			};
			\node[rectangle,inner sep=1mm] (G12) at (2,4) {
				\begin{tikzpicture}[baseline=(1.text)]
					\draw[thick] (0,0) circle (5mm) node (0) {$_4$};
					\node[draw,circle] (1) at (-0.66,0) {};
					\node[draw,circle] (2) at (0.6,0.33) {};
					\node[draw,circle] (3) at (0.6,-0.33) {};
					\node[rectangle] at (0,-1) {$G_{12} = \Span{s, tst^{-1}, t^{-1}st}$};
				\end{tikzpicture}
			};

			\draw[very thick] (G9) -- (G11);
			\draw[very thick] (G10) -- (G11);
			\draw[very thick] (G15) -- (G11);
			\draw[very thick] (G8) -- (G9);
			\draw[very thick] (G8) -- (G10);
			\draw[very thick] (G13) -- (G9);
			\draw[very thick] (G13) -- (G15);
			\draw[very thick] (G14) -- (G15);
			\draw[very thick] (G12) -- (G13);
			\draw[very thick] (G12) -- (G14);
		\end{tikzpicture}
	\end{center} 
\caption{Octahedral Family.}
\label{fig:octahedral}
\end{figure}

	%%%%%%%%%%%%%%%%%%%%%%%%%%%%%%%%%%%%%%%%%%%%%%%%%%%%%%%%%%%%%%%%%%%%%%%%%%%%%
The maximal group $G_{11}$ of the octahedral family has the presentation
	\begin{align*}
		G_{11} = \Span{s, t, u \mid s^2 = t^3 = u^4 = 1,\,
			stu = tus = ust}\text.
	\end{align*}
	We set again $z: = stu$.  Then
	\begin{align*}
		Z_{11} = \Span{z} = \{1, z,\dots, z^{23}\}
	\end{align*}
	is the centre of $G_{11}$.  Also set
	\begin{align*}
		U = \Span{u} = \{1, u, u^2, u^3\}\text,
	\end{align*}
	the parabolic subgroup of $G_{11}$ generated by $u$. This time we don't use the same parabolic group for all cases. More precisely, we have $G_j'=U$ for $j\in \{8,\dots, 15\}\setminus\{12,13,14,15\}$, since the generator $u$ does not appear in the presentation of the groups $G_{12}$, $G_{13}$, $G_{14}$ and $G_{15}$.
	\smallbreak\smallbreak\smallbreak  \noindent
\textbf{The group} $\mathbf{G_9.}$
	$$G_9 = \Span{s, u} \leq G_{11}.$$  Using $t^3 = 1$ it follows that
	\begin{align*}
		z^3 = (stu)^3 = sususu = ususus\text.
	\end{align*}
	Note that $Z_9 = \Span{z^3} = \{1, z^3, z^6, \dots, z^{21}\}$ is the centre of $G_9$.  We choose
	\begin{align*}
		X_9 = \{1, s, su, sus, su^2, su^2s\}\text.
	\end{align*}
	Then
	\begin{align*}
		G_9 = Z_9 U X_9 = \{yvx: y \in Z_9,\, v \in U,\, x \in X_9\}.
	\end{align*}
$X_9$ is the vertex set of the following spanning tree:
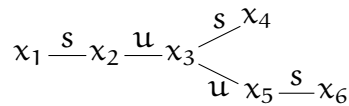
\begin{figure}[htb]
	\begin{center}
		\begin{tikzpicture}
			\node (1) at (0,0) {$x_1$};
			\node (2) at (1,0) {$x_2$};
			\node (3) at (2,0) {$x_3$};
			\node (4) at (3,0.5) {$x_4$};
			\node (5) at (3,-0.5) {\,$x_5$};
			\node (6) at (4,-0.5) {$x_6$};

			\draw (1) -- node[above] {$s$} (2);
			\draw (2) -- node[above] {$u$} (3);
			\draw (3) -- node[above] {$s$} (4);
			\draw (3) -- node[above, yshift=-0.4cm] {$u$} (5);
			\draw (5) -- node[above] {$s$} (6);
		\end{tikzpicture}
	\end{center} 
\end{figure}
		\smallbreak\smallbreak\smallbreak  \noindent
	\textbf{The group} $\mathbf{G_{11}.}$
	$$G_{11} = \Span{s, t, u}.$$  The set $\{1, z, z^2\}$
	is a transversal of the cosets of $G_9$ in $G_{11}$, i.e.,
	$G_{11} = G_9 \cup z G_9 \cup z^2 G_9$.  From $Z_{11} = Z_9 \cup z Z_9 \cup z^2 Z_9$ it follows that
	\begin{align*}
		G_{11} = Z_{11} U X_{11}\text,
	\end{align*}
	where $X_{11} = X_9 = \{1, s, su, sus, su^2, su^2s\}$.
		\smallbreak\smallbreak\smallbreak  \noindent
	\textbf{The group} $\mathbf{G_{10}.}$
	$$G_{10} = \Span{t, u} \leq G_{11}.$$ Using $s^2 = 1$, it follows that $sus = tut^{-1}$ and
	\begin{align*}
		z^2 = (stu)^2 = tutu = utut\text.
	\end{align*}
	Note that $Z_{10} = \Span{z^2} = \{1, z^2, z^4, \dots, z^{22}\}$ is the centre of $G_{10}$.  It follows that
	\begin{align*}
		G_{10} = Z_{10} U X_{10}\text,
	\end{align*}
	where
	\begin{align*}
		X_{10} = (X_{11} \cup z X_{11}) \cap G_{10}
		&= \{1, sus, su^2s, stu \cdot s, stu \cdot su, stu \cdot su^2 \}\\\notag
		&= \{1, tut^{-1}, tu^2t^{-1}, tu, tu^2, tu^3 \}\text.
	\end{align*}

$X_{10}$ is the vertex set of a spanning tree if $tu$ and $t^{-1}$ are added as generators.
\begin{figure}[htb]
	\begin{center}
		%\begin{tikzpicture}
		%	\node (4) at (1,0) {$x_4$};
		%	\node (2) at (2,0.5) {$x_2$};
		%	\node (5) at (2,-0.5) {$x_5$};
		%	\node (3) at (3,0) {$x_3$};
		%	\node (6) at (3,-1) {$x_6$};

		%	\draw (1) -- node[above] {$tu$} (4);
		%	\draw (4) -- node[above] {$t^{-1}$} (2);
		%	\draw (4) -- node[above] {$u$} (5);
		%	\draw (5) -- node[above] {$t^{-1}$} (3);
		%	\draw (5) -- node[above] {$u$} (6);
		%\end{tikzpicture}
		%\qquad
		\begin{tikzpicture}
			\node (1) at (0,0) {$x_1$};
			\node (4) at (1,0) {$x_4$};
			\node (2) at (2,0) {$x_2$};
			\node (5) at (3,0) {$x_5$};
			\node (3) at (4,0) {$x_3$};
			\node (6) at (5,0) {$x_6$};

			\draw (1) -- node[above] {$tu$} (4);
			\draw (4) -- node[above] {$t^{-1}$} (2);
			\draw (2) -- node[above] {$tu$} (5);
			\draw (5) -- node[above] {$t^{-1}$} (3);
			\draw (3) -- node[above] {$tu$} (6);
		\end{tikzpicture}
	\end{center} 
\end{figure}
\smallbreak\smallbreak\smallbreak  \noindent
\textbf{The group} $\mathbf{G_8.}$
	$$G_8 = \Span{u, v} \leq G_{11},$$ where $v = sus = tut^{-1}$.
	Set $w = z^3 s$.  Using $s^2 = t^3 = 1$, it follows that
	$w = (stu)^3s = uvu = vuv$.  Moreover, $s^{-1}ws =w$, $s^{-1}vs =u$ and
	\begin{align*}
		w^2 = uvuvuv = vuvuvu = (stu)^6\text.
	\end{align*}
	Note that $Z_8 = \Span{z^6} = \Span{w^2}$ is the centre of $G_8$.  It follows that
	\begin{align*}
		G_8 = Z_8 U X_8\text,
	\end{align*}
	where
	\begin{align*}
		X_8 = (X_9 \cup z^3 X_9) \cap G_8
		&= \{1, sus, su^2s,
		sususu \cdot s,
		sususu \cdot su,
		sususu \cdot su^2
		\}\\\notag
		&= \{1, v, v^2, w, wu, wu^2\}\text.
	\end{align*}
$X_8$ is the vertex set of a spanning tree if $w$ is added as a generator.

\begin{figure}[htb]
	\begin{center}
		\begin{tikzpicture}
			\node (1) at (0,0) {$x_1$};
			\node (2) at (1,0.5) {$x_2$};
			\node (3) at (2,0.5) {$x_3$};
			\node (4) at (1,-0.5) {\,$x_4$};
			\node (5) at (2,-0.5) {$x_5$};
			\node (6) at (3,-0.5) {$x_6$};

			\draw (1) -- node[above] {$v$} (2);
			\draw (2) -- node[above] {$v$} (3);
			\draw (1) -- node[above, yshift=-0.44cm] {$w$} (4);
			\draw (4) -- node[above] {$u$} (5);
			\draw (5) -- node[above] {$u$} (6);
		\end{tikzpicture}
	\end{center} 
\end{figure}
		\smallbreak\smallbreak\smallbreak  \noindent
	\textbf{The group} $\mathbf{G_{15}.}$
	$$G_{15} = \Span{s, t, u^2} \leq G_{11}.$$  It follows that $u^2st = st u^2$ and
	\begin{align*}
		z^2 = (stu)^2 = tu^2sts = tstu^2s = u^2stst = stu^2st = ststu^2\text.
	\end{align*}
	Note that $Z_{15} = \Span{z^2}$ is the centre of $G_{15}$.  Also note that
	$u^{-1} s u = s t s t^{-1} s^{-1}$.
	It follows that
	\begin{align*}
		G_{15} = Z_{15} \Span{u^2} X_{15}\text,
	\end{align*}
	where
	$$\begin{array}{lcl}
		X_{15}
		&=& (X_{11} \cup u^{-1}X_{11} \cup z X_{11} \cup u^{-1} z X_{11}) \cap G_{15}
		\\\notag
		& =& \{
		1, s, u^{-1}su, u^{-1}sus, su^2, su^2s,
		u^{-1}z, u^{-1}zs, zsu, zsus, u^{-1}zsu^2, u^{-1}zsu^2s
		\} \\\notag
		& \stackrel{u^{-1}z = st}{=}&\{
		1, s, stst^{-1}s^{-1}, stst^{-1}, su^2, su^2s,
		st, sts, ststst^{-1}s^{-1}, ststst^{-1}, stsu^2, stsu^2s
		\}\text.
	\end{array}$$
$X_{15}$ is the vertex of a spanning tree if  $st$, $s^{-1}$ and $t^{-1}$ are added as generators.
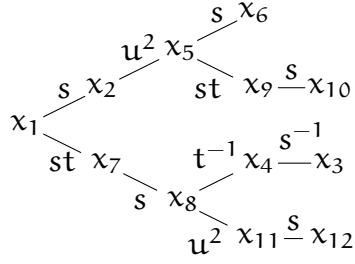
\begin{figure}[htb]
	\begin{center}
		\begin{tikzpicture}
			\node (1) at (0,0) {$x_1$};
			\node (2) at (1,0.5) {$x_2$};
			\node (3) at (4,-0.5) {$x_3$};
			\node (4) at (3,-0.5) {\,\,$x_4$};
			\node (5) at (2,1) {\,$x_5$};
			\node (6) at (3,1.5) {$x_6$};
			\node (7) at (1,-0.5) {\,\,$x_7$};
			\node (8) at (2,-1) {\,\,$x_8$};
			\node (9) at (3,0.5) {\,\,$x_9$};
			\node (10) at (4,0.5) {$x_{10}$};
			\node (11) at (3,-1.5) {\,\,$x_{11}$};
			\node (12) at (4,-1.5) {\,$x_{12}$};

			\draw (1) -- node[above] {$s$} (2);
			\draw (1) -- node[above,yshift=-0.47cm]{$st$}	(7);
				\draw (2) -- node[above] {$u^2$} (5);
			\draw (7) -- node[above,yshift=-0.47cm] {$s$} (8);
			\draw (5) -- node[above] {$s$} (6);
			\draw (5) -- node[above,yshift=-0.5cm, xshift=-0.1cm] {$st$} (9);
			\draw (8) -- node[above] {$t^{-1}$} (4);
			\draw (8) -- node[above,xshift=-0.1cm, yshift=-0.6cm] {$u^2$} (11);
			\draw (9) -- node[above] {$s$} (10);
			\draw (4) -- node[above,xshift=0.1cm] {$s^{-1}$} (3);
			\draw (11) -- node[above] {$s$} (12);
		\end{tikzpicture}
	\end{center} 
\end{figure}
\smallbreak\smallbreak\smallbreak  \noindent
\textbf{The group} $\mathbf{G_{13}.}$
	$$G_{13} = \Span{s, u^2, x} \leq G_{11},$$ where
	$x = t^{-1}s t = u s u^{-1}$.
	Using $t^3 = 1$, set $q:= u^{-1}z^3 = s x u^2 s = x u^2 s x$.
	It follows that
	$qx = sq$, and $u^2 q = q u^2$, and that
	\begin{align*}
		u^2 q^2 = z^6 = (stu)^6 =  (s x u^2)^3 = (x u^2 s)^3\text.
	\end{align*}
	Note that $Z_{13} = \Span{z^6}$ is the centre of $G_{13}$.  It follows that
	\begin{align*}
		G_{13} = Z_{13} \Span{u^2} X_{13}\text,
	\end{align*}
	where
	\begin{align*}
		X_{13} &= (X_{11} \cup u X_{11} \cup z^3 X_{11} \cup u^{-1} z^3 X_{11}) \cap G_{13} \\
		&= \{
		1, s, usu, usus, su^2, su^2s,
		u^{-1}z^3, u^{-1}z^3s, z^3su, z^3sus, u^{-1}z^3su^2, u^{-1}z^3su^2s
		\} \\
		&= \{
		1, s, qx^{-1}s^{-1}, qx^{-1}, su^2, su^2s,
		q, qs, su^2q, su^2qs, qsu^2, qsu^2s
		\}\text.
	\end{align*}
$X_{13}$ is the vertex set of a spanning tree if $q$, $s^{-1}$ and $x^{-1}$ are added as generators.
	
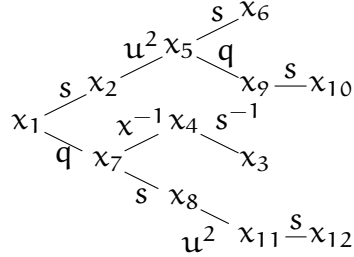
\begin{figure}[htb]
		\begin{center}
			\begin{tikzpicture}
				\node (1) at (0,0) {$x_1$};
				\node (2) at (1,0.5) {$x_2$};
				\node (3) at (3,-0.5) {$x_3$};
				\node (4) at (2,0) {\,\,$x_4$};
				\node (5) at (2,1) {$x_5$};
				\node (6) at (3,1.5) {$x_6$};
				\node (7) at (1,-0.5) {\,\,$x_7$};
				\node (8) at (2,-1) {\,\,$x_8$};
				\node (9) at (3,0.5) {$x_9$};
				\node (10) at (4,0.5) {$x_{10}$};
				\node (11) at (3,-1.5) {\,\,$x_{11}$};
				\node (12) at (4,-1.5) {$x_{12}$};

				\draw (1) -- node[above] {$s$} (2);
				\draw (1) -- node[below] {$q$} (7);
				\draw (2) -- node[above] {$u^2$} (5);
				\draw (7) -- node[below] {$s$} (8);
				\draw (5) -- node[above] {$s$} (6);
				\draw (5) -- node[above,pos=0.7] {$q$} (9);
				\draw (7) -- node[above, yshift=-0.1cm] {$x^{-1}$} (4);
				\draw (8) -- node[below, xshift=-0.2cm] {$u^2$} (11);
				\draw (9) -- node[above] {$s$} (10);
				\draw (4) -- node[above,xshift=0.25cm] {$s^{-1}$} (3);
				\draw (11) -- node[above] {$s$} (12);
			\end{tikzpicture}
		\end{center} 
	\end{figure}

	\smallbreak\smallbreak\smallbreak  \noindent
	\textbf{The group} $\mathbf{G_{14}.}$
	$$G_{14} = \Span{s, t} \leq G_{11}.$$  Using $u^4 = 1$ it follows that
	\begin{align*}
		z^4 = (stu)^4 = stststst = tstststs
	\end{align*}
	Note that $Z_{14} = \Span{z^4}$ is the centre of $G_{14}$.
	Then, with respect to the parabolic subgroup $\Span{t} = \{1, t, t^2\}$, we have
	\begin{align*}
		G_{14} = Z_{14} \Span{t} X_{14}\text,
	\end{align*}
	 where
	\begin{align*}
		X_{14} = \{
		1, s, st, sts, stst, ststs, stst^2, stst^2s
		\}\text.
	\end{align*}

$X_{14}$ is the vertex set of a spanning tree:

	\begin{figure}[htb]
		\begin{center}
			\begin{tikzpicture}
				\node (1) at (0,0) {$x_1$};
				\node (2) at (1,0) {$x_2$};
				\node (3) at (2,0) {$x_3$};
				\node (4) at (3,0) {$x_4$};
				\node (5) at (4,0) {$x_5$};
				\node (6) at (5,0.5) {$x_6$};
				\node (7) at (5,-0.5) {\,$x_7$};
				\node (8) at (6,-0.5) {$x_8$};

					\draw (1) -- node[above] {$s$} (2);
			\draw (2) -- node[above] {$t$} (3);
				\draw (3) -- node[above] {$s$} (4);
				\draw (4) -- node[above] {$t$} (5);
				\draw (5) -- node[above] {$s$} (6);
					\draw (5) -- node[below] {$t$} (7);
					\draw (7) -- node[above] {$s$} (8);
			\end{tikzpicture}
		\end{center} 
	\end{figure}

	%%%%%%%%%%%%%%%%%%%%%%%%%%%%%%%%%%%%%%%%%%%%%%%%%%%%%%%%%%%%%%%%%%%%%%%%%%%%%

	\smallbreak\smallbreak\smallbreak  \noindent
	\textbf{The group} $\mathbf{G_{12}.}$

	$$G_{12} = \Span{s, g, h} \leq G_{11},$$
	where
	$g = tst^{-1}$ and $h = t^{-1}st$.  Using $t^3 = 1$, set
	$y = sghs = ghsg = hsgh$.  Then $y^3 = (sgh)^4$.
	Note that $Z_{12} = \Span{z^{12}} = \Span{y^3}$ is the centre of $G_{12}$.
	We choose $Y = \{1, g, gh, ghg\}$. 	Then, with respect to the parabolic subgroup $\Span{s}$, we have
	\begin{align*}
		G_{12} = Z_{12} \Span{s} X_{12}\text,
	\end{align*}
	where
	\begin{align*}
		X_{12} &= Y \cup Y y \cup Y y^2 \\\notag
		&= \{
		1, g, gh, ghg, \\\notag
		& \phantom{=\{\}}
		ghsg, ghsgh, ghsghs, ghsghsh, \\\notag
		& \phantom{=\{\}}
		ghsghsgh, ghsghsghs, ghsghsghsg, ghsghsghsgs
		\}\text.
	\end{align*}

	$X_{12}$ is the vertex set of a spanning tree if $y$ is added as a generator:

				\begin{figure}[htb]
				\begin{center}
					\begin{tikzpicture}
						\node (1) at (0,0) {$x_1$};
						\node (5) at (1,-0.8) {$x_5$};
						\node (9) at (2,-0.8) {$x_9$};
						\node (2) at (1,0.5) {$x_2$};
						\node (3) at (2,1) {$x_3$};
						\node (6) at (2,-0.1) {$x_6$};
						\node(10) at (3,-0.1) {$x_{10}$};
						\node (4) at (3,1.5) {$x_4$};
						\node (8) at (4,1.5) {$x_8$};
						\node (12) at (5,1.5) {$x_{12}$};
						\node (7) at (3,0.6) {$x_7$};
						\node(11) at (4,0.6){$x_{11}$};

							\draw (1) -- node[above] {$g$} (2);
								\draw (2) -- node[above] {$h$} (3);
									\draw (3) -- node[above] {$g$} (4);
					\draw (4) -- node[above] {$y$} (8);
					\draw (8) -- node[above] {$y$} (12);
						\draw (1) -- node[below] {$y$} (5);
							\draw (5) -- node[below] {$y$} (9);
								\draw (2) -- node[below] {$y$} (6);
									\draw (6) -- node[below] {$y$} (10);
		\draw (3) -- node[below] {$y$} (7);
			\draw (7) -- node[below] {$y$} (11);
					\end{tikzpicture}
				\end{center} 
			\end{figure}

\subsection{$z$-bases and coset tables}We will now use the factorization $G_j=Z_jG_j'X_j$, $j=4,\dots, 15$, as an inspiration to construct a $z_j$-basis for each $H(G_j)$.
		We recall that $b\mapsto T_b$ denotes the restriction of the natural surjection $R(W)[B(W)]\rightarrow H(W)$ to $B(W)$, and that $\mathbold{s}$, $\mathbold{t}$, $\mathbold{u}$ are the braided reflections corresponding to $s, t, u$, respectively.	We denote with $\mathbold{G}_j'$, $\mathbold{Z}_j$, and $\mathbold{X}_j$ the sets that correspond to the sets $G_j'$, $Z_j$, and  $X_j$ defined before, where we replace each  word of letters $s$, $t$, $u$  with the letters $T_{\mathbold{s}}$, $T_{\mathbold{t}}$ and $T_{\mathbold{u}}$, respectively.
		 %Using this replacement, we also denote by $\mathbold{z}_j$ the element of $B(G_j)$, which corresponds to the element $z_j$.

The goal of this section is to prove that $\mathcal{B}_j=\mathbold{Z}_j\mathbold{G}_j'\mathbold{X}_j$ is a basis of $H(G_j)$. By construction, $\mathcal{B}_j$ is then a  $z_j$-basis of $H(G_j)$.

As we have seen in the previous section, the parabolic subgroup $G_j'$ is generated by a generator of $G_j$, which we denote by $g_0$.
We denote by $H'(G_j)$ the subalgebra of $H(G_j)$ generated by $T_{\mathbold{g_0}}$. Our goal is to prove that  $\mathcal{B}'_j:=\mathbold{Z}_j\mathbold{X}_j$ is a basis of $H(G_j)$ as $H'(G_j)$-module (and, therefore, $\mathcal{B}_j$ is a basis of $H(G_j)$ as $R(G_j)$-module).

Since $H(G_j)$ is a free $H'(G_j)$-module of dimension $d_j:=|G_j/G_j'|$ (see, for example, \cite{CC}), we only have to prove that $\mathcal{B}'_j$ is a spanning set for $H(G_j)$. By construction, we always have $1_{H(G_j)}\in \mathcal{B}'_j$ and, hence, it is enough to prove that   $\mathcal{B}'_j\cdot T_{\mathbold{g}}$ is a linear combination of elements in $\mathcal{B}'_j$, for each generator $g$ of $G_j$.

For this purpose, we construct a \emph{coset table}, as we also did in \cite{ChP} where we list this linear combination  for the elements $b_i.T_{\mathbold{g}}$, where $\mathcal{B}'_j=\{b_1,\dots, b_{d_j}\}$.
Each entry $b_i.T_{\mathbold{g}} = \sum_l \gamma_l b_l$ is a consequence of the defining relations of $H(G_j)$.

From the completed coset table we obtain  a faithful matrix representation of $H(G_j)$, which allows us to automate calculations inside the Hecke algebra. We refer to \cite[Proposition 2.5]{ChP}, where the reader can find an example of such representation and how it can be used in order to express every
element of $H(G_j)$ as an $R(G_j)$-linear combination of elements in the basis $\mathcal{B}_j$.

Filling the coset table is a step-by-step process, where new entries
are derived from existing entries and the relations by trial and
error.  In this process it is convenient to fill and use as known
entries additional columns for the elements
$b_i.T_{\mathbold{g}}^{-1}$, $b_i.T_{\mathbold{z_j}}$,
$b_i.T_{\mathbold{z_j}}^{-1}$, and for the elements $b_i.T_{\mathbold{g'}}$ and $b_i.T_{\mathbold{g'}}^{-1}$, where $g'$ are the redundant generators used in the spanning trees in Section~\ref{sec:factorization}.

We give in detail the example of $G_6$, where the reader can see thoroughly our methodology.

\begin{ex}
$G_6 = \Span{s, u \mid s^2 = u^4 = 1, sususu = ususus}$. As we have seen in Section \ref{tet},
$z_6 = sususu = ususus$ with $z_6^4=1$,
$G_6' = \Span{u} = \{1, u, u^2, u^3\}$ and $X_6=\{1, s, su, sus\}$.

To make notation lighter, we denote again with $s$ and $u$  generators of the Hecke algebra $H(G_6)$ (instead of $T_{\mathbf{s}}$ and $T_{\mathbf{u}}$) and with $z_6$ the element $T_{\mathbf{z_6}}$.
		With this new notation, we have then
		$$\mathcal{B}'_6=\{1,z_6,z_6^2,z_6^3\}\,\{1, s, su, sus\}=\{b_1,\,b_2,\,\dots,b_{16}\},$$
		where the elements $b_i$ are given explicitly in the following coset table. Recall that the entries of the coset table are $H'(G_6)$-linear combinations, where $H'(G_6)$ denotes the subalgebra of $H(G_6)$, generated by $u$.
		\[
		\begin{array}{|l|cccccc|}
			\hline
			b_i,\,i=1,\dots,16 & b_i.s & b_i.u &  b_i.z_6 & b_i.s^{-1} & b_i.u^{-1} & b_i.z_6^{-1}\\
			\hline
			b_1 = 1  & \underline{b_2} & u \cdot b_1 & b_5 & b_1:s^{-1} & u^{-1} \cdot b_1 &  \\
			b_2 =  s  & b_2:s & \underline{b_3} & b_6 & \underline{b_1} & \textcolor{red}{b_2:u^{-1}} & \\
			b_3 =  su & \underline{b_4} & \textcolor{red}{b_3.s^{-1}u^{-1}s^{-1}u^{-1}s^{-1}z_6} & b_7 & b_3:s^{-1} & \underline{b_2} & \\
			b_4 = sus  & b_4:s & \textcolor{blue}{b_4.s^{-1}u^{-1}s^{-1}u^{-1}s^{-1}z_6}  & b_8 & \underline{b_3} & \textcolor{blue}{b_4:u^{-1}} & \\
			\hline
			b_5 = z_6  & \underline{b_6} & u \cdot b_5 & b_9 & b_5:s^{-1} & u^{-1} \cdot b_5 & b_1 \\
			b_6 = z_6 s  & b_6:s & \underline{b_7} & b_{10} & \underline{b_5} & \textcolor{blue}{b_6.z_6^{-1}susus} & b_2 \\
			b_7 = z_6su & \underline{b_8} & \textcolor{blue}{b_7:u} & b_{11} & b_7:s^{-1} & \underline{b_6} & b_3 \\
			b_8 = zsus  & b_8:s & \textcolor{blue}{b_8.s^{-1}u^{-1}s^{-1}u^{-1}s^{-1}z_6} & b_{12} & \underline{b_7} & \textcolor{blue}{b_8:u^{-1}} & b_4 \\
			\hline
			b_9 = z_6^2  & \underline{b_{10}} & u \cdot b_9 & b_{13} & b_9:s^{-1} & u^{-1} \cdot b_9 & b_5 \\
			b_{10} = z_6^2 s  & b_{10}:s & \underline{b_{11}} & b_{14} & \underline{b_9} & \textcolor{blue}{b_9.z_6^{-1}susus} & b_6 \\
			b_{11} = z_6^2 su & \underline{b_{12}} & \textcolor{blue}{b_{11}:u} & b_{15} & b_{11}:s^{-1} & \underline{b_{10}} & b_7\\
			b_{12} = z_6^2 sus  & b_{12}:s & \textcolor{blue}{b_{12}.s^{-1}u^{-1}s^{-1}u^{-1}s^{-1}z_6} & b_{16} & \underline{b_{11}} & \textcolor{blue}{b_{12}:u^{-1}} & b_8 \\
			\hline
			b_{13} = z_6^3  & \underline{b_{14}} & u \cdot b_{13} & & b_{13}:s^{-1} & u^{-1} \cdot b_{13} & b_9 \\
			b_{14} = z_6^3 s  & b_{14}:s & \underline{b_{15}} & & \underline{b_{13}} & \textcolor{blue}{b_{14}.z_6^{-1}susus} & b_{10} \\
			b_{15} = z_6^3 su & \underline{b_{16}} & \textcolor{blue}{b_{15}:u} & & b_{15}:s^{-1} & \underline{b_{14}} & b_{11}\\
			b_{16} = z_6^3 sus  & b_{16}:s & \textcolor{red}{b_{16}:u} & & \underline{b_{15}} & \textcolor{red}{b_{16}.z_6^{-1}susus} & b_{12} \\
			\hline
		\end{array}
		\]
		The coset table is completed as follows. Notice here that each step we describe needs the previous ones, in order to be completed.
		\begin{itemize}
			\item First, the black entries are completed, which are straightforward: The entries $\underline{b_i}$ are part of the spanning tree, the entries $b_i$, $u\cdot b_i$ and $u^{-1}\cdot b_i$ are obtained from the definition of $b_i$
			and the entries $b_i:s$ and $b_i:s^{-1}$ can be computed from other black entries in the table if we apply the positive
			Hecke relation \ref{ph} and the inverse
			Hecke relation \ref{invhecke}  to the algebra generator $s$, respectively: e.g. $b_1.s^{-1}=b_1.(a_{s,0}^{-1}s-a_{s,0}^{-1}a_{s,1})=a_{s,0}^{-1}\cdot b_1.s-a_{s,0}^{-1}a_{s,1}\cdot b_1=a_{s,0}^{-1}\cdot b_2-a_{s,0}^{-1}a_{s,1}\cdot b_1$. \\
			\item The \textcolor{blue}{blue} entries follow next. The entries $b_i:u$ and $b_i:u^{-1}$ can be computed again using the positive
			Hecke relation \ref{ph} and the inverse
			Hecke relation \ref{invhecke}  to the algebra generator $u$. The  entries $b_i.u$ and $b_i.u^{-1}$ are obtained as an application of the relation $z_6=u. susus$
                        and, hence, $b_i. u=b_i.s^{-1}u^{-1}s^{-1}u^{-1}s^{-1}z_6$ and $b_i. u^{-1}=b_i.z_6^{-1}susus$.
                        The coset table at this point contains all the information needed to express these products as a linear combination: e.g.,
                        $$\begin{array}{lcl}
                        	b_4.u
                        = (b_4.s^{-1})u^{-1}s^{-1}u^{-1}s^{-1}z_6
                        &=& (b_3.u^{-1})s^{-1}u^{-1}s^{-1}z_6\smallbreak\smallbreak\\
                        &=& (b_2.s^{-1})u^{-1}s^{-1}z_6\smallbreak\smallbreak\\
                        &=&( b_1.u^{-1})s^{-1}z_6\smallbreak\smallbreak\\
                        &=& u^{-1} \cdot (b_1.s^{-1})z_6\smallbreak\smallbreak\\

                        &=&       u^{-1} \cdot (a_{s,0}^{-1}b_2-a_{s,0}^{-1}a_{s,1}b_1)z_6\smallbreak\smallbreak\\&=&     a_{s,0}^{-1}u^{-1}\cdot b_2.z_6- a_{s,0}^{-1}a_{s,1}u^{-1}\cdot b_1.z_6\smallbreak\smallbreak\\
                        &=& a_{s,0}^{-1}u^{-1}\cdot b_6- a_{s,0}^{-1}a_{s,1}u^{-1}\cdot b_5.
                        \end{array}$$
Here, we used directly the black entry $b_1.s^{-1}$ we have calculated as an example before.\\
                      \item The \textcolor{red}{red} entries come next, completed similarly as the blue ones.\\
			\item At this point, the first two columns of the coset table are complete, which means that $\mathcal{B}'_6$ is an $H'(G_6)$-basis of $H(G_6)$. As a result, it is not necessary to fill the missing entries. One could do that, using  any relation of the form $z_6 = w$ and the filled entries of the coset table.
		\end{itemize}

%	\begin{center}
%\begin{tikzpicture}
%												\def \n {4}
%												\foreach \s in {1,5,...,16}
%												\node (\s) at ({180 - (90/\n * (\s-1))}:3.25) {$_{\s}$};
%												\foreach \s in {2,6,...,16}
%												\node (\s) at ({180 - (90/\n * (\s-2))}:2.25) {$_{\s}$};
%												\foreach \s in {3,7,...,16}
%												\node (\s) at ({157.5 - (90/\n * (\s-3))}:1.5) {$_{\s}$};
%												\foreach \s in {4,8,...,16}
%												\node (\s) at ({112.5 - (90/\n * (\s-4))}:1.5) {$_{\s}$};
%
%%													\pgfmathtruncatemacro{\t}{1+\s}
%													\draw[g,very thick] (\s) -- (\t);
%													\pgfmathtruncatemacro{\a}{2+\s}
%													\draw[r,very thick] (\t) -- (\a);
%													\pgfmathtruncatemacro{\b}{3+\s}
%													\draw[g,very thick] (\a) -- (\b);
%												}

%%%													\pgfmathtruncatemacro{\b}{5+\s}
	%												\pgfmathtruncatemacro{\c}{6+\s}
	%%%												\fill[magenta!30,opacity=0.2] (\a.centre) -- (\b.centre) -- (\c.centre) -- cycle;
											%	}
											%\end{tikzpicture}
										%\end{center} 
										\qed
	\end{ex}

We finish this section by describing an extra property of the aforementioned $z_j$-basis $\mathcal{B}_j$:

We first consider the case of the tetrahedral family, i.e. $j=4,5,6,7$.  Since the BMR freeness conjecture holds, we have the following result \cite[Proposition 4.2 \& Table 4.6]{Ma2}: For each $j=4,5,6$, there is a specialisation $\theta_j: R(G_7)\rightarrow R(G_j)$, such that the Hecke algebra $H(G_j)$ is isomorphic to a subalgebra of the specialised Hecke algebra $H^j(G_7):=H(G_7)\otimes_{\theta_j}R(G_j)$. The specialisations $\theta_j$, $j=4,5,6$, are given explicitly in \cite[\S 4]{Ma2}.

	Recall that $\ell_j=|G_7:G_j|$ and that $z_j=z^{\ell_j}$, where $z$ is the generator of the centre of $G_7$ of order $12$.
	In the specialized algebra $H^j(G_7)$, we have
	that $T_{\mathbf{z}}^{\ell_j} = T_{\mathbf{z_j}}$ is an element of $H(G_j)$. Thus, as $H(G_j)$-algebra,
	\begin{equation}\label{sp}H^j(G_7) = H(G_j) + T_{\mathbf{z_j}} H(G_j) + \dots +
	T_{\mathbf{z_j}}^{\ell_j-1} H(G_j).
	\end{equation}

	We consider now the $z$-basis of $H(G_7)$, which is of the form $\mathcal{B}_7 =\{1, T_{\mathbf{z}},\dots,T_{\mathbf{z}}^{11}\}Y,$ for some set $Y$. Then,
	\begin{equation}\label{sp2}
	\mathcal{B}_j = \{1, T_{\mathbf{z}}^{\ell_j}, \dots, T_{\mathbf{z}}^{(k-1)\ell_j}\}  Y',	\end{equation}
	where $z_j^k=1$, and
	$Y' = \{1, T_{\mathbf{z}}^{\ell_j}, \dots, T_{\mathbf{z}}^{(k-1)\ell_j}\}  Y'  \cap H(G_j),$
		where, for each $y\in Y$, exactly one of the $T_{\mathbf{z}}^{\alpha} y$ lies in $Y'$.

	There is an analogous property for the octahedral family, i.e. $j=8,\dots, 15$, where the role of $G_7$ is played by the maximal group of this family, namely the group $G_{11}$.

\subsection{The Gram matrix}
	As we have seen so far, for each complex reflection group $G_j$, $j=4,\dots, 15$ we have constructed a $z_j$-basis $\mathcal{B}_j$. To make notation lighter, we will denote again with $z_j$ the algebra element $T_{\mathbf{z_j}}$.
	Recall that $m=|G_j/Z(G_j)|$. Let $y_i \in B(W)$, $i = 1 \dots, l$,  be such that $\mathcal{B}_j$ is of the form
	$$\mathcal{B}_j = \{z_j^k y_i\,:\, k = 0,\dots,m-1,\, i = 1,\dots,l\}.$$

	The goal of this section is to calculate the Gram matrix

	\begin{align*}
		A_j = (\tau_{\mathcal{B}_j}(z^{k_1} y_{i_1}\, z^{k_2} y_{i_2})) = (\tau_{\mathcal{B}_j}(z^{k_1 + k_2} y_{i_1} y_{i_2})),
	\end{align*}
which corresponds to the linear map $\tau_{\mathcal{B}_j}$ and prove it is symmetric.

\subsubsection{Calculation and symmetry} From the definition of the $z_j$-basis, it follows that
$A_j$
	is an $m \times m$-block matrix
	\begin{align*}
		A_j = \left(
		\begin{array}{cccc}
			A_j^0&A_j^1&\dots&A_j^{m-1}\\
			A_j^1&A_j^2&\dots&A_j^m \\
			\dots & \dots && \dots \\
			A_j^{m-1} & A_j^m & \dots & A_j^{2m-2}
		\end{array}
		\right),
	\end{align*}
	where for $k_1,k_2\in \{1,\dots, m-1\}$
	\begin{align*}
		A_j^{k_1+k_2}: = (\tau_{\mathcal{B}_j}((z^{k_1+k_2} y_{i_1} y_{i_2}))_{i_1,i_2=1}^l.
	\end{align*}

The way we have constructed the basis $\mathcal{B}_j$,  many products $y_{i_1}y_{i_2}$ are elements inside $\{1=y_1,\dots, y_l\}$. Therefore, the matrices $A_j^0, \dots, A_j^{m-1}$ contain many zeros and are relatively easy to compute.

For the remaining matrices $A_j^m, \dots, A_j^{2m-2}$ one could use the closed coset tables, in order to calculate each entry of the matrices, a method we have already seen in \cite{ChP}. This time  we use a recursion formula to speed up the procedure, as follows.

We first consider the elements $z_j^my_i$, $i=1,\dots, l$. These elements are not inside $\mathcal{B}_j$, but we could write them as a linear combination of elements in $\mathcal{B}_j$:
	\begin{align*}
	z_j^m y_i
	%  = \sum_{j=1}^{lm} \zeta_{ij} y_j
	= \sum_{p=0}^{m-1} \sum_{q=1}^{l} \zeta^p_{iq} z_j^p y_q,\,\,\,\, i=1,\dots, l.
\end{align*}

We use now the coefficients  $\zeta^p_{iq}\in R(G_j)$ to construct the following $l \times lm$-matrix:
\begin{align*}
	Z: = \left(
	\begin{array}{cccc}
		Z^0 & Z^1 & \dots & Z^{m-1}
	\end{array}
	\right) = (\zeta^p_{iq})
\end{align*}
\begin{lem}\label{gram}
	Let $\alpha\in \{0,\dots, m-2\}$. We have the following recursion formula
	\begin{align*}
		A_j^{\alpha+m} = Z \cdot \left(
		\begin{array}{cccc}
			A_j^{\alpha} & A_j^{\alpha+1} & \dots & A_j^{\alpha+m-1}
		\end{array}
		\right)^T.
	\end{align*}
\end{lem}
	\begin{proof}
		$\begin{array}[t]{lcl}
			(A_j^{\alpha+m})_{i_1,i_2} = \tau_{\mathcal{B}_j}(z_j^{\alpha+m} y_{i_1} y_{i_2})&=&\tau_{\mathcal{B}_j}(z_j^my_{i_1}\cdot z_j^{\alpha}y_{i_2})\smallbreak\smallbreak\\
			&=& \tau_{\mathcal{B}_j}\left((\sum\limits_{p=0}^{m-1} \sum\limits_{q=1}^l \zeta^p_{i_1q}z_j^p y_q) z_j^{\alpha} y_{i_2}\right)\smallbreak\smallbreak\\
			&= &\sum\limits_{p=0}^{m-1} \sum\limits_{q=1}^l \zeta^p_{i_1q}\, \tau_{\mathcal{B}_j}(z_j^{\alpha+p} y_{q} y_{i_2})\smallbreak\smallbreak\\
			&=& \sum\limits_{p=0}^{m-1} \sum\limits_{q=1}^l Z^p_{i_1,q}  (A_j^{\alpha+p})_{q,i_2}\smallbreak\smallbreak \\
			&= &\sum\limits_{p=0}^{m-1} (Z^p \,A_j^{\alpha+p})_{i_1,i_2}.
		\end{array}$

	\end{proof}
Using the inductive construction of the block form of the Gram matrix, one could prove the following result:
\begin{cor}If the matrices $A_j^0,\dots, A_j^{m-1}$ are symmetric, then the matrix $A_j$ is symmetric.
	\end{cor}
	\begin{proof}
		Let $\alpha \in\{0,\dots, m-2\}$. As we have seen in the proof of Lemma \ref{gram}, $$(A_j^{\alpha+m})_{i_1,i_2}=\sum_{p=0}^{m-1} \sum_{q=1}^l Z^p_{i_1,q}  (A_j^{\alpha+p})_{q,i_2}.$$
		On the other hand, we have:
		$$\begin{array}{lcl}
			(A_j^{\alpha+m})_{i_2,i_1} = \tau(z_j^{\alpha+m} y_{i_2} y_{i_1})&=&\tau_{\mathcal{B}_j}(z_j^{\alpha}y_{i_2}\cdot z_j^my_{i_1})\smallbreak\smallbreak\\
			&=& \tau_{\mathcal{B}_j}\left(z_j^{\alpha} y_{i_2}\cdot \sum\limits_{p=0}^{m-1} \sum\limits_{q=1}^l \zeta^p_{i_1q}z_j^p y_q \right)\smallbreak\smallbreak \\
			&=& \sum\limits_{p=0}^{m-1} \sum\limits_{q=1}^l \zeta^p_{i_1q}\, \tau_{\mathcal{B}_j}(z_j^{\alpha+p} y_{i_2} y_q)\smallbreak\smallbreak\\
			&=& \sum\limits_{p=0}^{m-1} \sum\limits_{q=1}^l Z^p_{i_1,q}  (A_j^{\alpha+p})_{i_2,q}.
		\end{array}$$
	Therefore, if  $A_j^{\alpha+p}$ is symmetric for every $p\in\{0,\dots, m-1\}$, the matrix $A_j^{\alpha+m}$ is also symmetric. We use this fact to  prove that $A_j^{\alpha+m}$ is symmetric, for every $\alpha\in \{0,\dots, m-2\}$ (and, hence, the Gram matrix $A$ is also symmetric).
	\begin{itemize}
		\item For $\alpha=0$: $A_j^m$ is symmetric, since $A_j^0,\dots, A_j^{m-1}$ are symmetric.
		\item For $\alpha=1$: $A_j^{1+m}$ is symmetric, since $A_j^{1}, \dots, A_j^m$ are symmetric.
		\item For $\alpha=2$: $A_j^{2+m}$ is symmetric, since $A_j^{2}, \dots, A_j^{m},\, A_j^{m+1}$ are symmetric.\\
		\phantom{For $\alpha=2$: $A_j^{2+m}$ is symmetric,}
	 $\vdots$
	 \item For $\alpha=m-2$: $A_j^{2m-2}$ is symmetric, since $A_j^{m-2}, \dots, A_j^{2m-3}$ are symmetric.
 \end{itemize}
\end{proof}

These techniques allow us to computationally verify the following.

\begin{Proposition}\label{tr}
  For $j \in \{4,\dots, 15\}$, the Gram matrix $A_j$ of the linear map $\tau_{\mathcal{B}_j}$ of $H(G_j)$ is symmetric, i.e., $\tau_{\mathcal{B}_j}$ is a  trace function.
\end{Proposition}

\subsubsection{The determinant} In order to prove that $\tau_{\mathcal{B}_j}$ is indeed a
symmetrising trace, we need to show that the determinant of
$A_j$, $j=4,\dots, 15$ is invertible in the definition ring $R(G_j)$.  \smallbreak\smallbreak
\noindent
	\textbf{Tetrahedral family.}
	We first set up some notation for the coefficients of the positive Hecke relations \ref{ph} for the maximal group $G_7$:
	\begin{itemize}
		\item $T_{\mathbf{s}}^2 = a_1 T_{\mathbf{s}} + a_0$
		\item $T_{\mathbf{t}}^3 = b_2 T_{\mathbf{t}}^2 + b_1 T_{\mathbf{s}}+ b_0$
		\item $T_{\mathbf{u}}^3 = c_2T_{\mathbf{u}}^2 + c_1T_{\mathbf{u}}  + c_0$
	\end{itemize}
	Recall that the coefficients $a_0$, $b_0$, and $c_0$ are invertible elements of $R(G_7)$.
	As we have already mentioned, for each $j=4,5,6$, there is a specialisation $\theta_j: R(G_7)\rightarrow R(G_j)$, such that the Hecke algebra $H(G_j)$ is isomorphic to a subalgebra of the specialised Hecke algebra $H^j(G_7)$. The specialisations $\theta_j$, $j=4,5,6$, are given explicitly in \cite[\S 4]{Ma2}. For the coefficients $a_i$, $b_i$, and $c_i$ in the positive Hecke relations we have:

	$$
	\begin{array}{l|c|c|c|c|c|c|c|c}
		\theta_j,\, j=4,5,6 & \theta_j(a_0) & \theta_j(a_1) &  \theta_j(b_0) & \theta_j(b_1) & \theta_j(b_2) &\theta_j(c_0) & \theta_j(c_1) & \theta_j(c_2)\\
		\hline
		\theta_4  & 1 & 0 &  1& 0 & 0 &c_0 & c_1 & c_2\\
		\theta_5  & 1 & 0 &  b_0& b_1 & b_2 &c_0 & c_1 & c_2\\
		\theta_6  & a_0 & a_1 &  1& 0 & 0 &c_0 & c_1 & c_2\\
			\end{array}$$
		\\
		The determinants $\det A_j$, $j=4,5,6,7$ are given in the following table:

		\begin{align*}
			\setlength\extrarowheight{2.5pt}
			\begin{array}{l|l}
				G_j & \det A_j \\ \hline
				G_7  & a_0^{936} b_0^{528} c_0^{672}\\
				G_6  & a_0^{264} c_0^{192}\\
				G_5  & -b_0^{264} c_0^{336}\\
				G_4  & -c_0^{96}
			\end{array}
		\end{align*}

\smallbreak \noindent
	\textbf{Octahedral family.}
	There is an analogue for the octahedral family. We start again with the positive Hecke relations of the maximal group $G_{11}$:
	\begin{itemize}
		\item $s^2 = a_1 s + a_0$
		\item $t^3 = b_2 t^2 + b_1 t + b_0$
		\item $u^4 = c_3 u^3 + c_2 u^2 + c_1 u + c_0$
	\end{itemize}
	The determinants $\det A_j$, $j=8,\dots, 15$ are given in the following table:
	\begin{align*}
		\setlength\extrarowheight{2.5pt}
		\begin{array}{l|l}
			G_j & \det A_j \\ \hline
			G_{11}  & a_0^{7296} b_0^{4416} c_0^{4032}\\
			G_9 & a_0^{2240} c_0^{1248}\\
			G_{10}  & b_0^{2208} c_0^{2016}\\
			G_{15}& a_0^{3648} b_0^{2208} c_0^{2016}\\
			G_8  & c_0^{624}\\
			G_{13} & a_0^{1120} c_0^{592} \\
			G_{14}  & -a_0^{1692} b_0^{1200} \\
			G_{12}  & -a_0^{576}
		\end{array}
	\end{align*}
	Therefore, we have proven the following:

	\begin{Proposition}\label{tr2}
		For $j \in \{4,\dots, 15\}$ the trace function $\tau_{\mathcal{B}_j}$ is a symmetrising  trace.
	\end{Proposition}

There is an alternative way proving that $\det A_j$, $j=4,\dots, 15$ is invertible, without calculating all determinants explicitly:

\begin{lem}\phantom{AA}
	\begin{enumerate}
	\item Let $A_{7}^j$, $j=4,5,6$ be the Gram matrix, which corresponds to the specialised Hecke algebra $H^j(G_7)$. Then, we have the following block form:
	\[
	A_7^j= \left[\begin{array}{ c  c cc}
		A_j & &&\\

		 &&& D_2\\
			&&\iddots&\\
			& D_{\ell_j}&&
	\end{array}\right]
	\]
	Therefore, if $\det A_7^j$ is invertible in $R(G_j)$, the same stands for $\det A_j$.
	\item Let $A_{11}^j$, $j=8,9,10,12,13,14,15$ be the Gram matrix, which corresponds to the specialised Hecke algebra $H^j(G_{11})$. Then, we have the following block form:
	\[
	A_{11}^j= \left[\begin{array}{ c  c cc}
		A_j & &&\\

		 &&& D_2\\
			&&\iddots&\\
			& D_{\ell_j}&&
	\end{array}\right]
	\]
	Therefore, if $\det A_{11}^j$ is invertible in $R(G_j)$, the same stands for $\det A_j$.
		\end{enumerate}
\end{lem}
\begin{proof} We prove (1), i.e. we prove the claim for the tetrahedral family. The proof for the octahedral family is similar, if we replace $G_7$ with $G_{11}$.

$H^j(G_7)$ admits a basis $B_7$, coming from the $z$-basis $\mathcal{B}_7$ with extension of scalars. From \eqref{sp} we have that any element in $\mathcal{B}_7$  has the form $T_{\mathbf{z}}^{\alpha} b$ for some
	$b\in B_j$, and $0 \leq \alpha < \ell_j$.   Multiplying any two gives

	$$T_{\mathbf{z}}^{\alpha}b\cdot T_{\mathbf{z}}^{\alpha'}b'  = T_{\mathbf{z}}^{\alpha+\alpha'}bb'.$$ From \eqref{sp2} we have, then that $T_{\mathbf{z}}^{\alpha}b\cdot T_{\mathbf{z}}^{\alpha'}b'$
	is $0$ under $\tau_{\mathcal{B}_j}$, unless  $\alpha+\alpha'$ is a multiple of $\ell_j$.
		\end{proof}

	\subsection{Conclusion}
	In this section we summarize our results and we present the main theorem of this paper.

	We consider the exceptional groups of the tetrahedral and octahedral family, i.e. the groups $G_j$, $j=4,\dots, 15$. For each group we have constructed a $z_j$-basis $\mathcal{B}_j$ and we have defined the linear form $\tau_{\mathcal{B}_j}: H(G_j)\rightarrow R(G_j)$, as $1_{H(G_j)}\mapsto 1$ and $b\mapsto 0$, for $b\not =1_{H(G_j)}$.

	Our goal is to prove that $\tau_{\mathcal{B}_j}$ is the canonical symmetrising  trace on $H(G_j)$ of the BMM symmetrising trace conjecture \ref{BMM sym}.

	 In the previous section we have seen that $\tau_{\mathcal{B}_j}$ satisfies Condition (1) of Conjecture \ref{BMM sym} (Propositions \ref{tr} and \ref{tr2}). From Remark \ref{remex} we have that $\mathcal{B}_j$ satisfies the lifting conjecture \ref{lift} and, hence,
	  Condition (2) is satisfied and Condition (3) is equivalent to \ref{extra2}.
We check this condition by using the closed coset table:

Let $B_j=\{z_j^ky_i\,:\, k=0,\dots, m-1,\, i=1,\dots, l\}$ the $z_j$-basis of $H(G_j)$. One needs to prove that $\tau_{\mathcal{B}_j}(z_j^{m-k}y_i^{-1})=0$, for all $(k,i)\not=(0,1)$. Using the matrix representation coming from the closed coset table we can write each element $z_j^{m-k}y_i^{-1}$ as a linear combination of elements of $\mathcal{B}_j$, with the coefficient of $1_{H(G_j)}$ being its image under $\tau_{\mathcal{B}_j}$. We verify that this coefficient is always 0 and, therefore, we  obtain the following:

	  \begin{Theorem}
	  	The BMM symmetrising trace conjecture \ref{BMM sym} holds for the Hecke algebras associated to the groups of the tetrahedral and octahedral families.
	  \end{Theorem}

	\end{document}